\documentclass[english]{article}
\usepackage[utf8]{inputenc}
\usepackage{babel}
\usepackage{units}
\usepackage{amstext,amsmath}
\usepackage{amssymb}
\usepackage{graphicx}
\usepackage{hyperref}
\usepackage{amsthm}
\usepackage{color}
\usepackage{dsfont}
\usepackage[export]{adjustbox}
\usepackage{subfig}
\usepackage{tikz}
\usetikzlibrary{calc}
\usepackage{etoolbox}
\usepackage{pgfplots}
\pgfplotsset{compat=newest}
\usepackage[shortlabels]{enumitem}
\usepackage{todonotes}
\usepackage[capitalize]{cleveref}
\usepackage{booktabs}
\usepackage{soul}
\usepackage{forloop}
\pgfplotsset{every tick label/.append style={font=\tiny}}
\pgfplotsset{
	tick label style = {font = {\fontsize{6 pt}{12 pt}\selectfont}},
	label style = {font = {\fontsize{6 pt}{12 pt}\selectfont}},
	legend style = {font = {\fontsize{6 pt}{12 pt}\selectfont}},
}

\makeatletter
\newtheorem{thm}{Theorem}
\theoremstyle{plain}
\theoremstyle{remark}
\newtheorem{rem}[thm]{Remark}
\theoremstyle{plain}
\newtheorem{lem}[thm]{Lemma}
\theoremstyle{plain}
\newtheorem{prop}[thm]{Proposition}
\theoremstyle{plain}

\theoremstyle{definition}
\newtheorem{definition}[thm]{Definition}

\newcommand{\R}{\mathds{R}}
\newcommand{\N}{\mathds{N}}
\newcommand{\calI}{\mathcal{I}}
\newcommand{\C}{\mathds{C}}
\newcommand{\T}{T}
\newcommand{\K}{\mathds{K}}
\newcommand{\Z}{\mathds{Z}}
\newcommand{\calM}{\mathcal{M}}

\newcommand{\calK}{\mathcal{K}}

\newcommand{\norm}[1]{\left\lVert#1\right\rVert}
\newcommand{\OmGrid}{\Omega_\#}
\newcommand{\xGrid}{x_\#}
\newcommand{\vGrid}{v_\#}
\renewcommand{\epsilon}{\varepsilon}
\DeclareMathOperator*{\argmin}{arg\,min}
\DeclareMathOperator*{\sgn}{sgn}

\newcommand{\unaryminus}{{\text{\scalebox{0.4}[1.0]{\( - \)}}}}

\newlength\figureheight
\newlength\figurewidth

\newtoggle{usetikz}
\toggletrue{usetikz}
\newcommand{\InputImage}[3]{}
\iftoggle{usetikz}{%
	\renewcommand{\InputImage}[3]{%
		\setlength\figureheight{#2}%
		\setlength\figurewidth{#1}%
		\input{figures/#3.tikz}%
	}%
} { %
	\renewcommand{\InputImage}[3]{%
		\includegraphics[width=#1]{figures/#3}%
	}%
}

\graphicspath{{figures/}}
\newcommand{\email}[1]{\protect\href{mailto:#1}{#1}}

\title{Dynamic Spike Super-resolution and Applications to Ultrafast Ultrasound Imaging}

\date{September 19, 2018}
        
        \author{Giovanni S.\ Alberti\thanks{Department of Mathematics, University of Genoa, Via Dodecaneso 35, 16146 Genova, Italy (\email{alberti@dima.unige.it}).}
\and Habib Ammari\thanks{Department of Mathematics,
ETH Z\"{u}rich, R\"{a}mistrasse 101, 8092 Z\"{u}rich, Switzerland (\email{habib.ammari@math.ethz.ch}, \email{francisco.romero@sam.math.ethz.ch}).}
\and Francisco Romero\footnotemark[2]
 \and Timoth\'ee Wintz\thanks{Sony CSL Paris,
6 rue Amyot, 75005 Paris, France (\email{timothee.wintz@sony.com}).}
}

\begin{document}
	\maketitle
\begin{abstract}
We consider the dynamical super-resolution problem consisting in the recovery of positions and velocities of moving particles from low-frequency static measurements taken over multiple time steps. The standard approach to this issue is a two-step process: first, at each time step some static reconstruction method is applied to locate the positions of the particles with super-resolution and, second, some tracking technique is applied to obtain the velocities. In this paper we propose a fully dynamical method based on a phase-space lifting of the positions and the velocities of the particles, which are simultaneously reconstructed with super-resolution. We provide a rigorous mathematical analysis of the recovery problem, both  for the noiseless case and in presence of noise (in the discrete setting). Several numerical simulations illustrate and validate our  method, which shows some  advantage over existing techniques. 
We then discuss the application of this  approach to the dynamical super-resolution problem in ultrafast ultrasound imaging: blood vessels' locations and blood flow velocities are recovered with super-resolution.
\end{abstract}	

\vspace{2mm}

	
	\section{Introduction}
	 It is well-known that the resolution of any wave imaging method is limited by the diffraction limit \cite{ammari2008introduction}. Super-resolution is understood as any technique whose resolution surpasses this fundamental limit, which is of order of half the operating wavelength. More precisely, the super-resolution problem can be stated as follows: given  low frequency measurements of a medium, typically obtained with a convolution by a low pass filter, reconstruct the original medium with a resolution exceeding the diffraction limit. 
 Since high frequencies are completely lost in the measuring process, it is  impossible to solve this problem in the general case.  It is then natural to focus on the particular case when the medium is made of a finite number of point sources, with unknown locations  and intensities.  This framework finds applications in many imaging modalities, from the pioneering work on super-resolved fluorescence microscopy \cite{Hell94,moerner1997,Betzig2006,hess2006ultra,thompson2012extending}  (Nobel Prize in Chemistry 2014 \cite{nobel}) to the works on super-focusing in locally resonant media \cite{ammari2015mathematical, ammari2015super, ammari2017sub, lemoult2013wave, lerosey2007focusing} and the more recent findings in ultrafast  ultrasound localization microscopy \cite{desailly2013sono,errico2015ultrafast}.

In mathematical terms,  each point source is represented as a Dirac delta $w_i \delta_{x_i}$, with unknown locations $x_i\in X\subseteq \R^d$ and intensities $w_i\in\C$.	 We have the sparse spike reconstruction problem: recover 
 \[
	\mu = \sum_{i=1}^N w_i \delta_{x_i}
	\]
	 from the measurements $y=\mathcal{F} \mu$, where \begin{equation} \label{eq:measurOp}
		\mathcal{F} : \mathcal{M}(X) \rightarrow \mathbb{R}^n
	\end{equation}
	is the measurement operator  from  the set of Radon measures  $\mathcal{M}(X)$ defined on $X$.
		 Since $\mathcal{M}(X)$ is infinite dimensional, $\mathcal{F}$ is not injective, and therefore one has to use regularization to recover $\mu$. The common choice in this context is an infinite dimensional variant of the $\ell^1$ minimization:
	\begin{equation}
		  \min_{\nu\in \mathcal{M}(X)} \left\Vert \nu \right\Vert_{\textrm{TV}} \quad\text{subject to $\mathcal{F}\nu = y$.}
		\label{eq:lasso}
	\end{equation}
	Mathematical theory on this problem has greatly advanced over the past years. It includes stable reconstruction of spikes with separation in one and multiple dimensions~\cite{candes2014towards,candes2013}, robust recovery of positive spikes in the case of a Gaussian point spread function with no condition of separation  \cite{bendory2015robust}, exact reconstruction for positive spikes in a general setting \cite{de2012exact} and the corresponding stability properties \cite{denoyelle2016support,poon2017multi}.

Whenever the dynamics of the medium  is relevant, as in the case of blood vessel imaging,
 the super-resolution problem for dynamic point reflectors $\mu_t$ arises. 
  In this case, at each time step one measures $\mathcal{F}\mu_t$ and needs to reconstruct both the locations of the spikes and their dynamics: we call this problem the dynamic spike super-resolution problem (see \Cref{fig:procedure}). The  current  approach for this problem is to perform  a static reconstruction at each time step (using the method discussed above), and then
	to track the spikes to obtain  their velocities \cite{errico2015ultrafast}. This approach suffers from three  main drawbacks: first,  a lot of data are discarded whenever static reconstruction cannot be performed because of particles being too close, second, the  information from neighboring frames is ignored in the first step of the reconstruction and, third,  tracking algorithms are computationally expensive. 
	
	In this paper we propose a new method for this dynamical super-resolution problem based on a fully dynamical inversion scheme, in which  the spikes' locations and velocities are simultaneously reconstructed. After assuming a (local) linear movement of the spikes, we lift the problem to the phase space, in which each spike becomes a particle with location and velocity. The super-resolution problem is then set in this augmented domain, and the minimization performed directly with the full dynamical data. We provide a theoretical investigation of this technique (the analysis shares some common aspects with the one presented in \cite{2017radial} for a similar problem, the main novelty lying in the  geometrical insights), including exact and stable recovery properties, as well as extensive numerical simulations. These simulations show the great potential of this approach for practical applications, far beyond the predictions of the theory, which we believe  can be further developed.

As mentioned above, one of the main motivations and applications of this work is ultrafast  ultrasound localization microscopy \cite{desailly2013sono,errico2015ultrafast}. Ultrafast ultrasonography is a recent imaging modality based on the use of plane waves instead of the usual focused waves \cite{montaldo-tanter-bercoff-benech-finck-2009,62014-tanter-fink,demene2015spatiotemporal}. 
The resolution of ultrafast ultrasound is determined by the wavelength of the incident wave, and by other factors such as the length of the receptor array and the range of angles used in angle compounding~\cite{alberti2017mathematical}. Due to diffraction theory, the highest resolution is half a wavelength, which is of the order of $\unit[300]{\mu m}$. Thus, in blood vessel imaging, blood vessels separated by less than $\unit[300]{\mu m}$  cannot be distinguished. As in fluorescence microscopy, 
randomly activated micro-bubbles in the blood may be used to produce very localized spikes in the observations, giving rise to a dynamical super-resolution problem \cite{errico2015ultrafast}, in which both the locations and the velocities of the bubbles are of interest (the velocities are also used to estimate the thickness of the blood vessels). This is a framework when our approach can be immediately applied.

	This paper is structured as follows. In \Cref{sec:setting} we describe the dynamical super-resolution problem and 
discuss the method introduced in this paper: the phase-space lifting. In \Cref{sec:exact,sec:other} we study the exact recovery issue in absence of noise,
while in \Cref{sec:stability} we prove a stability result for noisy measurements in the discrete setting. In \Cref{sec:numerical} we provide several numerical simulations which 
validate the method and the theoretical results. In \Cref{sec:ultrafast} we discuss the applications of this technique to ultrafast ultrasound localization microscopy. Finally, \Cref{sec:conclusion} contains some concluding remarks and future perspectives.

	\begin{figure}[h]
        \centering        
\captionsetup[subfigure]{width=\textwidth}
\subfloat[Position and velocity of the particles: the dots correspond to the positions of particles in the middle frame ($t=t_0$), and the arrows to the displacement of the particles from the first frame to the last one.]{\makebox[\textwidth]{
                \includegraphics[width=0.25\textwidth]{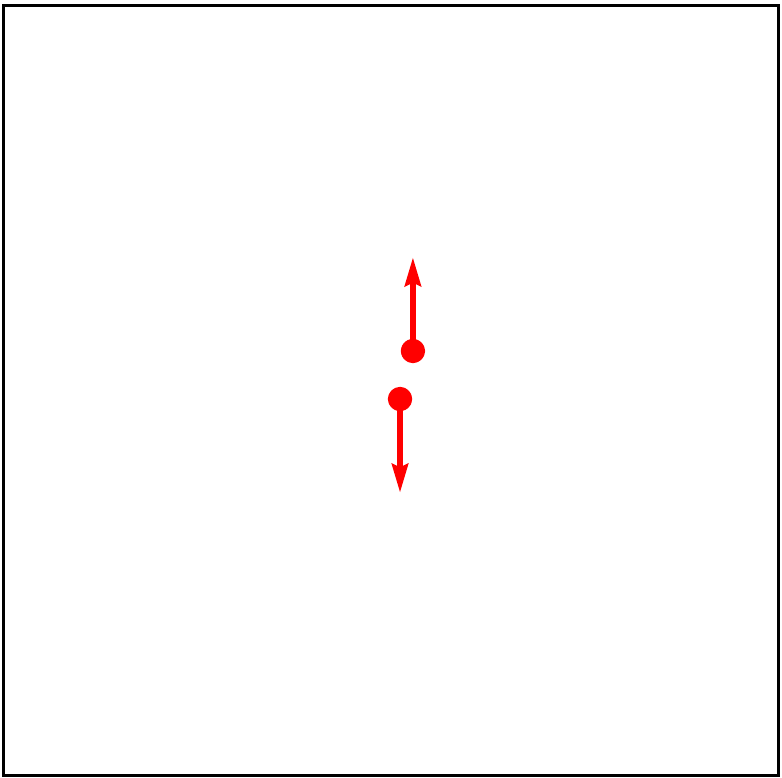}
            }}
            
        \subfloat[Corresponding measured sequence $\{y_k\}_k$.]{
                \includegraphics[width=0.19\textwidth]{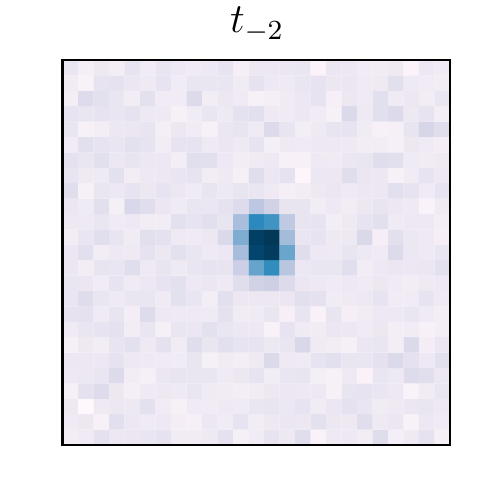}

                \includegraphics[width=0.19\textwidth]{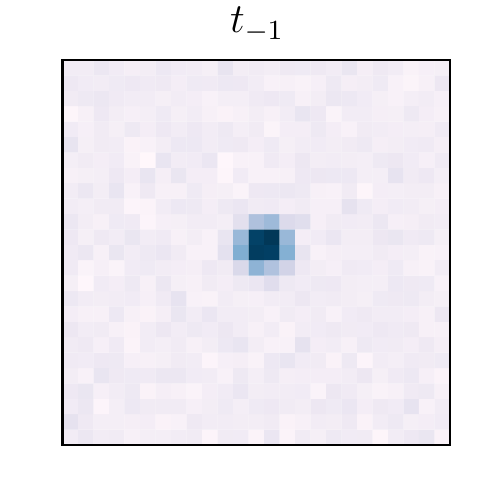}

                \includegraphics[width=0.19\textwidth]{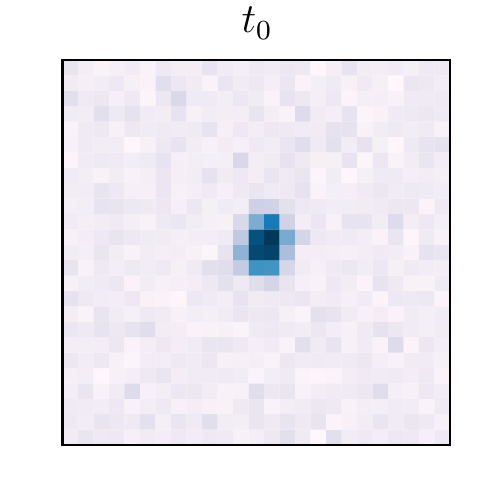}

                \includegraphics[width=0.19\textwidth]{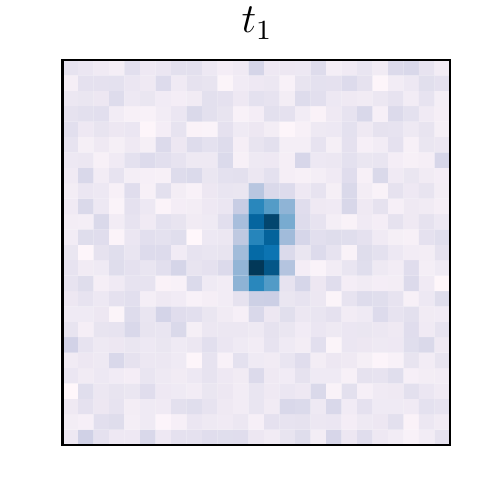}

                \includegraphics[width=0.19\textwidth]{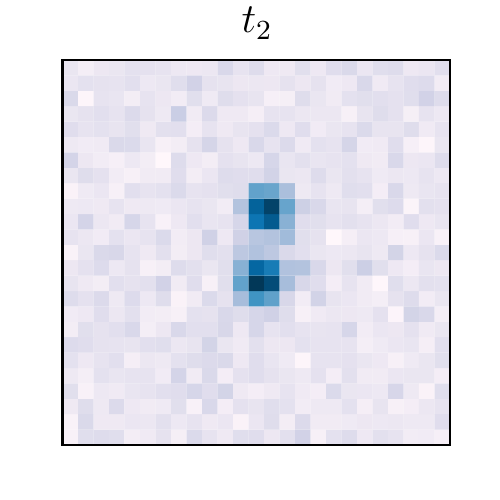}
        }
       
            \caption{Illustration of the dynamic spike super-resolution problem, with parameters $d=2$, $N=2$ and $K=2$ and $\mathcal{F}$ is a convolution operator by a Gaussian point spread function.}
        \label{fig:procedure}
    \end{figure}

    \section{Setting the stage}\label{sec:setting}
    \subsection{The space-velocity model}

Let us now introduce   our model for super-resolution of dynamic spikes. Instead of considering a single measure $\mu$, we consider a time-varying measure $\mu_t$, where $t\in [-\delta, \delta]$, and $\delta>0$ defines our observation window. Since $\delta$ is expected to be small, we can approximate the dynamics of each point linearly. Therefore we model each point source as a particle  displacing  with a constant velocity: 
    \[
        \mu_t = \sum_{i=1}^N w_i \delta_{x_i + v_i t}, \quad t\in [-\delta, \delta],
    \]
    where $v_i \in \mathbb{R}^d$. The measurement vector is then composed of uniform samples in the observation window  $t_k = k \tau$ for  $k \in \{-K,-K+1,\dots,K-1,K\}$, with  $K = \delta/\tau\in\N$ (where  $\N$ denotes the set of all positive integers):
    \[
        y_k = \mathcal{F}\mu_{t_k},\quad k\in [-K, K].
    \]
    Figure~\ref{fig:procedure} illustrates the space-velocity model in two dimensions.

    In this work, we show that under certain conditions, we are able to recover  the positions $x_i$, the velocities $v_i$ and the weights $w_i$ simultaneously with infinite precision, using a sparse spike recovery  method. 

    From now on we will assume that the phase space particles, understood as positions $x_i$ and velocities $v_i$, lie inside the domain
    \begin{equation} \label{eq:omega} \Omega = \left\{ (x,v) \in \R^{2d}: x + k \tau v \in [0,1]^d, \quad \forall k \in [-K,K] \right\}, \end{equation}
        that is the space-velocity domain in which, for all the considered time samples, the locations of the  particles  stay inside $[0,1]^d$. Let $\T = \{(x_i,v_i)\}_{i=1}^N$ denote the set of particles. Furthermore, the set of associated weights $w_i$ will be considered in $\K$, where $\K$ can be either $\C$ or $\R$.

    The measurement operator $\mathcal{F}\colon \mathcal{M}([0,1]^d)\to\C^n$ is  assumed to be of the form
    \[
        \mathcal{F} \nu = \left( \left\langle \nu, \varphi_l \right\rangle \right)_{l=1}^n,\qquad \nu\in \mathcal{M}([0,1]^d),
    \]
    where  $\left\{ \varphi_l \right\}$ is a family of test functions defined on $[0,1]^{d}$ and $\langle \nu, \varphi_l \rangle = \int_{[0,1]^d}\varphi_l\,d\nu$.   Applying $\mathcal{F}$ to the measures $\mu_{k}:=\mu_{t_k}$ for every time step $k\in \left\{-K,\dots,K\right\}$ gives
    \begin{equation*}
        \mathcal{F}\mu_{k}  = \left( \left< \mu_{k}, \varphi_l \right> \right)_{l=1}^{n}  = \left( \sum_{i=1}^{N} w_i \varphi_l (x_i + k \tau v_i ) \right)_{l=1}^{n}.
    \end{equation*}
    By construction, these measurements are composed of a vector of size $n$ for each time sample $k$. We now describe how to express these measurements via an operator defined on measures on the phase space. Consider the  measure $\omega\in\mathcal{M}(\Omega)$ and the   family of  test functions  $\varphi_{l,k}\in C(\Omega)$ given by 
    \begin{equation*} 
        \omega = \sum_{i=1}^N w_i \delta_{x_i, v_i}, \qquad 
        \varphi_{l,k}(x,v) = \varphi_l(x + k \tau v),
    \end{equation*}
    and the measurement operator
    \begin{equation} \label{eq:G}
        \mathcal{G} : \mathcal{M} \left(\Omega \right) \rightarrow \R^{n \times(2K +1)}, \qquad \mathcal{G} \lambda  = \left( \left< \lambda, \varphi_{l,k}\right> \right)_{l,k}.
    \end{equation}
    With these objects we can write $(\mathcal{F}\mu_k)_k = \mathcal{G}\omega$, and so  the measurements  are given by
    \begin{equation} \label{eq:signalModel}
        y = \mathcal{G} \omega.
    \end{equation}

The dynamical reconstruction problem is now set in the phase space, and consists in the recovery of the sparse measure $\omega$ from the measurements \eqref{eq:signalModel}. 
Following the approach for sparse spike recovery, we pose this inversion as a  total variation (TV) optimization problem, in which we seek to reconstruct positions and velocities simultaneously by minimizing
    \begin{equation}\label{eq:lasso_v}
        \min_{\lambda \in \mathcal{M}\left( \Omega \right)} \norm{\lambda}_{\textrm{TV}} \quad\text{subject to }\;   \mathcal{G}\lambda = y,
    \end{equation}
    where the TV norm of a measure $\lambda \in \mathcal{M}\left( \Omega \right)$ is defined as
    \[
    \norm{\lambda}_{\textrm{TV}}:=\sup\left\{
    \int_\Omega f\,d\lambda:f\in C(\Omega),\; \norm{f}_\infty\le 1\right\}.
    \]
We will call \eqref{eq:lasso_v} the \emph{dynamical} recovery, whereas \eqref{eq:lasso} will be called the \emph{static} recovery. The aim of this section is to determine conditions under which exact recovery holds, namely when $\omega$ is the unique minimizer of \eqref{eq:lasso_v}. In this way, the  recovery of the measure $\omega$ from the data $y$ reduces to a convex optimization problem.

    \subsection{The perfect low-pass case} \label{subsec:lowFreq}
    Instead of studying the general framework outlined so far, in order to highlight the main features of this approach we prefer to focus on the particular case of low-frequency Fourier measurements, which represents a simplified model for many different  applications. Thus, the theoretical analysis discussed below refers only to this situation, even though most parts may be extended to the general case of a convolution operator.

    The low-frequency Fourier measurements are expressed by the complex sinusoids $\varphi_l(x)  = e^{- 2\pi i x\cdot l}$ for $l\in\Z^d$ with $\|l\|_\infty\le f_c$, where $f_c \in \N$ is the highest available frequency for the considered imaging system. The static measurements are given by
\[
    (\mathcal{F}\nu)_l = \int_{[0,1]^d} e^{- 2\pi i x\cdot l}\,d\nu(x),\qquad l\in\left\{-f_c,\dots,f_c\right\}^d.
\]	
With dynamical data, for $k\in\{-K,\dots,K\}$, we  have	$\varphi_{l,k}(x,v) = e^{ 2\pi i (x+k \tau v)\cdot l}$ and so
    \begin{equation}\label{eq:GFourier}
        (\mathcal{G}\omega)_{l,k} = \int_{\Omega} e^{- 2\pi i(x+k \tau v)\cdot l }\,d\omega(x,v)
        = \int_{\Omega} e^{- 2\pi i (x,v)\cdot(l,k \tau l) }\,d\omega(x,v).
    \end{equation}
This expression shows that our data consist of low-frequency samples of $2d$-dimensional Fourier measurements restricted to the $d$-dimensional subspaces $\{(\xi,k\tau\xi)\in\R^{2d}:\xi\in\R^d\}$ for $k\in\{-K,\dots,K\}$. Thus, our problem is in principle harder than the one considered in \cite{candes2014towards}, in which one measures all low-frequency Fourier coefficients, and not only
\begin{equation}\label{eq:butterflySet}
    \left\{(l, k\tau l):l\in\Z^d, \norm{l}_{\infty} \leq f_c,\;k=-K,\dots,K \right\},
\end{equation}
and different from the one considered in \cite{2017radial}, in which the Fourier transform is sampled along lines.
    For a visual representation of this restriction in the case  $d = 1$, see \Cref{fig:butterfly}. 

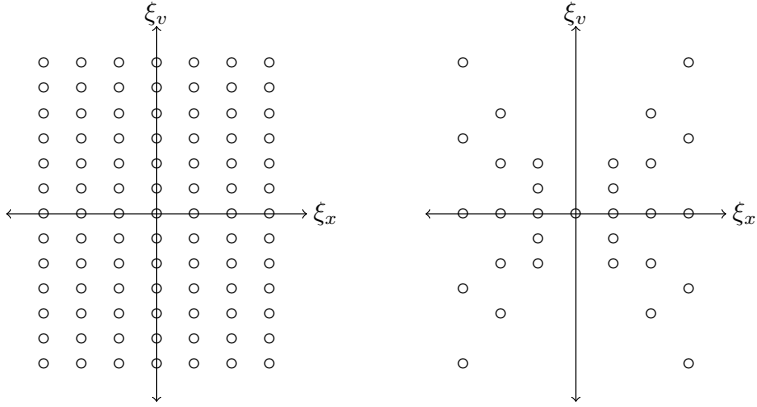
\begin{figure} \centering
    \begin{minipage}{0.45\linewidth}
        \begin{center}
            \begin{tikzpicture}
                \draw[<->] (0,-2.5) -- (0,2.5);
                \draw[<->] (-2,0) -- (2,0);
                \draw (-3/2,-6/3) node {$\circ$}; \draw (-3/2,-5/3) node {$\circ$};
                \draw (-3/2,-4/3) node {$\circ$}; \draw (-3/2,-3/3) node {$\circ$};
                \draw (-3/2,-2/3) node {$\circ$}; \draw (-3/2,-1/3) node {$\circ$};
                \draw (-3/2,-0/3) node {$\circ$}; \draw (-3/2,1/3) node {$\circ$};
                \draw (-3/2,2/3) node {$\circ$}; \draw (-3/2,3/3) node {$\circ$};
                \draw (-3/2,4/3) node {$\circ$}; \draw (-3/2,5/3) node {$\circ$};
                \draw (-3/2, 6/3) node {$\circ$};
                \draw (-2/2,-6/3) node {$\circ$}; \draw (-2/2,-5/3) node {$\circ$};
                \draw (-2/2,-4/3) node {$\circ$}; \draw (-2/2,-3/3) node {$\circ$};
                \draw (-2/2,-2/3) node {$\circ$}; \draw (-2/2,-1/3) node {$\circ$};
                \draw (-2/2,0/3) node {$\circ$}; \draw (-2/2,1/3) node {$\circ$};
                \draw (-2/2,2/3) node {$\circ$}; \draw (-2/2,3/3) node {$\circ$};
                \draw (-2/2,4/3) node {$\circ$}; \draw (-2/2,5/3) node {$\circ$};
                \draw (-2/2,6/3) node {$\circ$};
                \draw (-1/2,-6/3) node {$\circ$}; \draw (-1/2,-5/3) node {$\circ$};
                \draw (-1/2,-4/3) node {$\circ$}; \draw (-1/2,-3/3) node {$\circ$};
                \draw (-1/2,-2/3) node {$\circ$}; \draw (-1/2,-1/3) node {$\circ$};
                \draw (-1/2,-0/3) node {$\circ$}; \draw (-1/2,1/3) node {$\circ$};
                \draw (-1/2,2/3) node {$\circ$}; \draw (-1/2,3/3) node {$\circ$};
                \draw (-1/2,4/3) node {$\circ$}; \draw (-1/2,5/3) node {$\circ$};
                \draw (-1/2,6/3) node {$\circ$};
                \draw (0,-6/3) node {$\circ$}; \draw (0,-5/3) node {$\circ$};
                \draw (0,-4/3) node {$\circ$}; \draw (0,-3/3) node {$\circ$};
                \draw (0,-2/3) node {$\circ$}; \draw (0,-1/3) node {$\circ$};
                \draw (0,0/3) node {$\circ$}; \draw (0,1/3) node {$\circ$};
                \draw (0,2/3) node {$\circ$}; \draw (0,3/3) node {$\circ$};
                \draw (0,4/3) node {$\circ$}; \draw (0,5/3) node {$\circ$};
                \draw (0,6/3) node {$\circ$};
                \draw (1/2,-6/3) node {$\circ$}; \draw (1/2,-5/3) node {$\circ$};
                \draw (1/2,-4/3) node {$\circ$}; \draw (1/2,-3/3) node {$\circ$};
                \draw (1/2,-2/3) node {$\circ$}; \draw (1/2,-1/3) node {$\circ$};
                \draw (1/2,-0/3) node {$\circ$}; \draw (1/2,1/3) node {$\circ$};
                \draw (1/2,2/3) node {$\circ$}; \draw (1/2,3/3) node {$\circ$};
                \draw (1/2,4/3) node {$\circ$}; \draw (1/2,5/3) node {$\circ$};
                \draw (1/2,6/3) node {$\circ$};
                \draw (2/2,-6/3) node {$\circ$}; \draw (2/2,-5/3) node {$\circ$};
                \draw (2/2,-4/3) node {$\circ$}; \draw (2/2,-3/3) node {$\circ$};
                \draw (2/2,-2/3) node {$\circ$}; \draw (2/2,-1/3) node {$\circ$};
                \draw (2/2,0/3) node {$\circ$}; \draw (2/2,1/3) node {$\circ$};
                \draw (2/2,2/3) node {$\circ$}; \draw (2/2,3/3) node {$\circ$};
                \draw (2/2,4/3) node {$\circ$}; \draw (2/2,5/3) node {$\circ$};
                \draw (2/2,6/3) node {$\circ$};
                \draw (3/2,-6/3) node {$\circ$}; \draw (3/2,-5/3) node {$\circ$};
                \draw (3/2,-4/3) node {$\circ$}; \draw (3/2,-3/3) node {$\circ$};
                \draw (3/2,-2/3) node {$\circ$}; \draw (3/2,-1/3) node {$\circ$};
                \draw (3/2,-0/3) node {$\circ$}; \draw (3/2,1/3) node {$\circ$};
                \draw (3/2,2/3) node {$\circ$}; \draw (3/2,3/3) node {$\circ$};
                \draw (3/2,4/3) node {$\circ$}; \draw (3/2,5/3) node {$\circ$};
                \draw (3/2, 6/3) node {$\circ$};
                \draw (0,8/3) node {$\xi_v$};
                \draw (4.5/2,0) node {$\xi_x$};
            \end{tikzpicture}
        \end{center}
    \end{minipage}
        \begin{minipage}{0.45\linewidth}
            \begin{center}
                \begin{tikzpicture}
                    \draw[<->] (0,-2.5) -- (0,2.5);
                    \draw[<->] (-2,0) -- (2,0);
                    \draw (0,0) node {$\circ$};
                    \draw (1/2,0) node {$\circ$};
                    \draw (1/2,1/3) node {$\circ$};
                    \draw (1/2,2/3) node {$\circ$};
                    \draw (1/2,-1/3) node {$\circ$};
                    \draw (1/2,-2/3) node {$\circ$};
                    \draw (-1/2,0) node {$\circ$};
                    \draw (-1/2,1/3) node {$\circ$};
                    \draw (-1/2,2/3) node {$\circ$};
                    \draw (-1/2,-1/3) node {$\circ$};
                    \draw (-1/2,-2/3) node {$\circ$};
                    \draw (2/2,0) node {$\circ$};
                    \draw (2/2,2/3) node {$\circ$};
                    \draw (2/2,4/3) node {$\circ$};
                    \draw (2/2,-2/3) node {$\circ$};
                    \draw (2/2,-4/3) node {$\circ$};
                    \draw (-2/2,0) node {$\circ$};
                    \draw (-2/2,2/3) node {$\circ$};
                    \draw (-2/2,4/3) node {$\circ$};
                    \draw (-2/2,-2/3) node {$\circ$};
                    \draw (-2/2,-4/3) node {$\circ$};
                    \draw (3/2,0) node {$\circ$};
                    \draw (3/2,3/3) node {$\circ$};
                    \draw (3/2,6/3) node {$\circ$};
                    \draw (3/2,-1) node {$\circ$};
                    \draw (3/2,-2) node {$\circ$};
                    \draw (-3/2,0) node {$\circ$};
                    \draw (-3/2,1) node {$\circ$};
                    \draw (-3/2,2) node {$\circ$};
                    \draw (-3/2,-1) node {$\circ$};
                    \draw (-3/2,-2) node {$\circ$};
                    \draw (0,8/3) node {$\xi_v$};
                    \draw (4.5/2,0) node {$\xi_x$};
                \end{tikzpicture}
            \end{center}
        \end{minipage}
        \caption{The allowed frequencies for the full low-frequency case (left) and in our case, when $d=1$, $f_c = 3$ and $K = 2$.}
        \label{fig:butterfly}
\end{figure}

    \section{Exact recovery in absence of noise}\label{sec:exact}

    \subsection{Dual certificates}\label{sub:dual}

As it is standard in convex optimization, it is useful to  consider the dual problem to study the exact recovery for  \eqref{eq:lasso_v}.
In order to do this, we need to introduce the concept of dual certificate. We use the notation $\sgn \K^N = \{\eta\in\K^N:|\eta_i|=1\text{ for every $i=1,\dots,N$}\}$.

    \begin{definition}[Dual certificate]
        \label{def:dualCertificate}  Let $T=\{(x_i,v_i)\}_{i=1,\dots,N}\subseteq\Omega$ be a configuration of particles  and  $\eta \in \sgn\K^{N}$. A \emph{dual certificate} of the dynamical recovery problem \eqref{eq:lasso_v} is a function  
        \begin{equation}
            \label{eq:form}
            q(x, v) = \sum_{k=-K}^{K} \ \sum_{\norm{l}_\infty \leq f_c} c_{k,l} e^{i 2\pi l \cdot \left( x + k\tau v \right)},
        \end{equation}
        where $c_{k,l} \in \C$,		 obeying
        \begin{equation}
            \label{eq:cond}
            \left\{\begin{array}{ll} q(x_i, v_i) =
                \eta_i, & i\in\{1,\dots,N\}, \\
                |q(x, v)| < 1, & (x,v)\in \Omega \setminus \T.
            \end{array}\right.
        \end{equation}
        In the following, we shall say that a function of the form  \eqref{eq:form} has \emph{dynamical form}.
    \end{definition}

    In \Cref{fig:healthyDualcertificate}, we present an example of a dual certificate for the dynamical recovery problem, which was computed using a predefined kernel as in \cite{candes2014towards}. 

    The existence of a dual certificate guarantees exact recovery for the minimization problem \eqref{eq:lasso_v}. More precisely, we have the following result (for a proof, see \cite[Proposition A.1]{candes2014towards}).

    \begin{figure}
        \centering
        \includegraphics[width=0.5\textwidth]{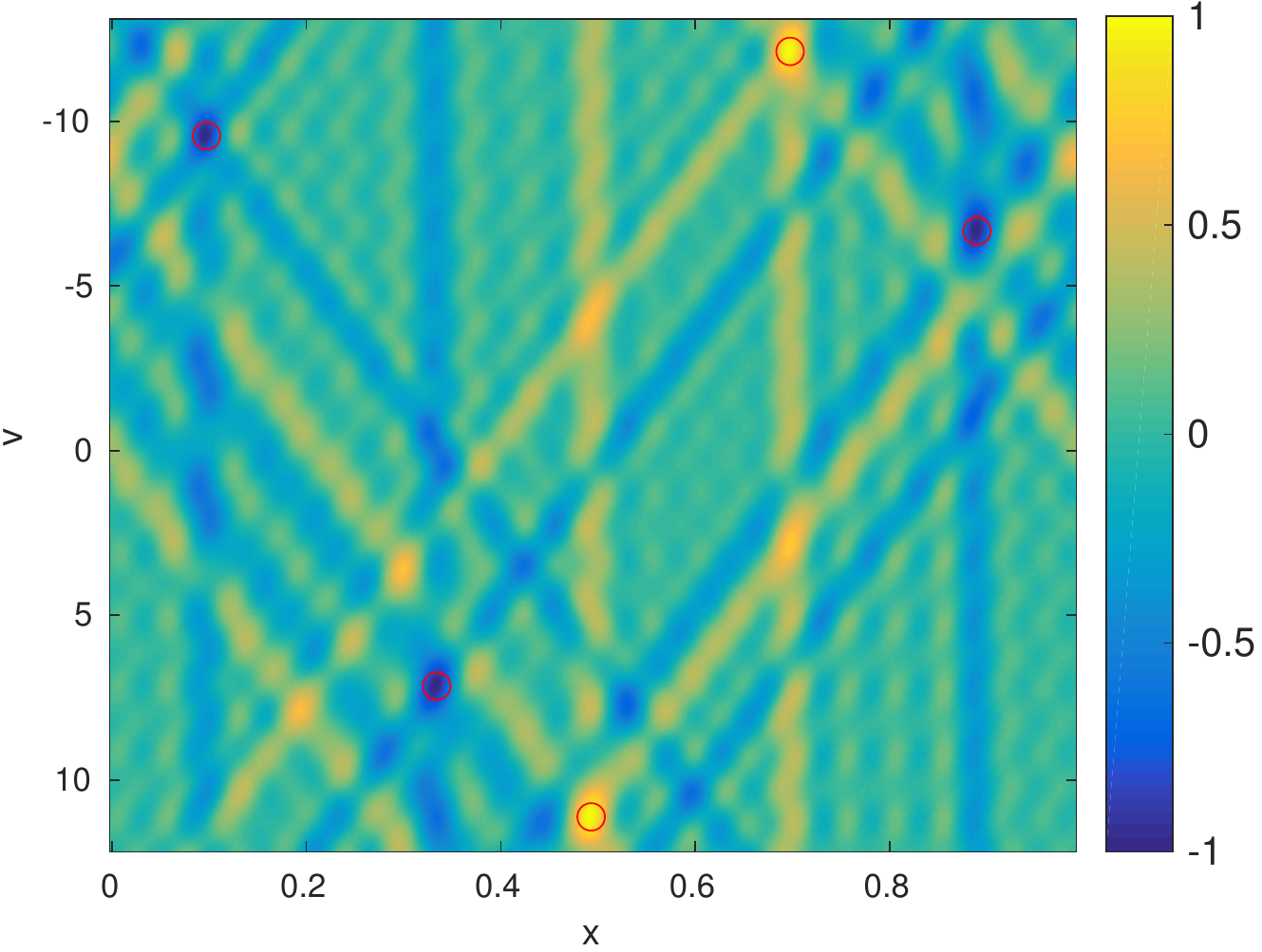}
        \caption{ Example in $d=1$ of a dual certificate for $N=5$ particles (inside the red circles) with real weights. The parameters are $K = 1$, $f_c = 20$ and $\tau = 0.025$. The positions, speeds and weights were selected at random. }
        \label{fig:healthyDualcertificate}
    \end{figure}

    \begin{prop} \label{prop:dual} Suppose that for every $\eta \in \sgn\K^N$ there exists a dual certificate for the dynamical recovery problem, and let $\hat\omega$ be a minimizer of \eqref{eq:lasso_v}. Then $\hat\omega=\omega$.
    \end{prop}

 With this proposition in hand, our problem reduces to finding conditions under which  dual certificates exist. As mentioned above, the static recovery problem was treated in \cite{candes2014towards}, but their methodology cannot be transferred directly to our case since  the static dual certificates are constructed with all low frequency coefficients, whereas in our case we have access only to the frequencies given by the set \eqref{eq:butterflySet}.

    The particular structure of functions with dynamical form \eqref{eq:form} allows for a simple decomposition 	
    \begin{equation*}
        q(x,v) = \frac{1}{|\calK|}\sum_{k\in\calK} q_k(x,v), \qquad q_k(x,v) =\sum_{\norm{l}_\infty \leq f_c} c_{k,l} e^{i 2\pi l \cdot \left( x + k\tau v \right)},
    \end{equation*}
    where $\calK\subseteq\{-K,\dots,K\}$ is a subset of the time samples that is used to construct the dual certificate.
    Observe that the functions $q_k$ are constant along the $d$-dimensional subspaces parallel to $\{ (x,v)\in\R^{2d} : x+k \tau v = 0 \}$. Thus, in principle they can be seen as functions of $[0,1]^d \subseteq \R^d$ instead of  $\Omega \subseteq \R^{2d}$. More precisely, we write
    \begin{equation*}
        q_k(x,v) = \tilde q_k(x+k \tau v),\qquad \tilde q_k(y) = \sum_{\norm{l}_\infty \leq f_c } c_{k,l} e^{i2\pi l \cdot y }.
    \end{equation*}
    Consider the values of these functions on the location of the particles at each time
    \begin{equation} \label{eq:gamma}
        \gamma_{i,k} := q_k(x_i,v_i) = \tilde q_k(x_i + k \tau v_i), \quad i \in \{1,\ldots,N\},\ k \in \{-K,\ldots,K\}.
    \end{equation}
    The functions $q_k(x,v)$ are constant along the affine spaces
    \begin{equation*} 
        L_{i,k} := \{ (x,v) \in \Omega: (x-x_i) + k \tau (v- v_i) = 0 \},
    \end{equation*}
    which contain the particle $i$. This implies that  the constants $\gamma_{i,k}$ propagate along them, namely
    \[
        q_k(x,v)=\gamma_{i,k},\qquad (x,v)\in L_{i,k}.
    \]
    In other words, the values  of the dual certificate on $L_{i,k}$ are completely determined by $\gamma_{i,k}$. 	Moreover, by \eqref{eq:cond}, these values must satisfy the conditions
        \begin{equation}\label{prop:gammas}
            \frac{1}{|\calK|}\sum_{k\in\calK} \gamma_{i,k} = \eta_i, \qquad  i\in \{1,\ldots,N\}.
        \end{equation}
 In \Cref{fig:ghost} we present an example in $d=1$, with three static particles ($N=3$) and three time measurements ($K=1$). The values of $q$ are fixed by $\gamma_{i,k}$ on each line $L_{i,k}$, and in particular in their points of intersection. As we shall see below, the problematic points are those where several lines (or $d$-dimensional affine subspaces) intersect, as in the two circled dots in the figure.

    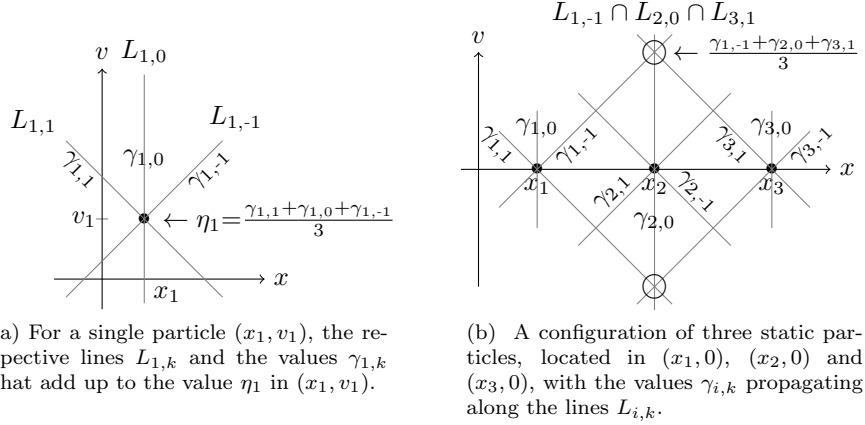
\begin{figure}
        \newcommand{\varDist}{2}
        \newcommand{\varSlope}{1}
        \newcommand{\varLinelength}{2.3}
        \newcommand{\varLinelengthShort}{1.3}
        \newcommand{\varXone}{1}
        \newcommand{\varXtwo}{\varXone+\varDist}
        \newcommand{\varXthree}{\varXone + 2*\varDist}
        \centering\captionsetup[subfigure]{width=150pt}
         \subfloat[For a single particle $(x_1,v_1)$, the respective lines $L_{1,k}$  and the values $\gamma_{1,k}$ that add up to the value $\eta_1$ in $(x_1,v_1)$.]{\label{fig:ghosta}
                \begin{tikzpicture}[scale = 0.8]
                    \draw (1,1) node[right] {$\ \leftarrow \eta_1 {=} \frac{\gamma_{1,1} {+} \gamma_{1,0} {+} \gamma_{1,\unaryminus 1}}{3}$} node {$\bullet$};
                    \draw (1+1.1,1+0.8) node [rotate=45] {$\gamma_{1,\unaryminus 1}$};
                    \draw (1-1.1,1+0.8) node [rotate=-45] {$\gamma_{1,1}$};
                    \draw (1,2) node {$\gamma_{1,0}$};
                    \draw (1,-0.25) node[right] {$x_1$};
                    \draw (-0,1) node {$v_1$};
                    \draw[-,gray] (0.2,1) -- (0.4,1);
                    \draw[->] (-0.5,0) -- (3,0);
                    \draw[->] (0.3,-0.5) -- (0.3, 3.5);
                    \draw (3,0) node[right] {$x$};
                    \draw (0.3,3.5) node[above] {$v$};
                    \draw[-,gray] (\varXone -\varLinelengthShort,\varLinelengthShort/\varSlope + 1) -- (\varXone +\varLinelengthShort, -\varLinelengthShort/\varSlope + 1);
                    \draw (\varXone -\varLinelengthShort,\varLinelengthShort/\varSlope + 1 + 0.3) node[left] {$ \quad L_{1,1}$};
                    \draw[-,gray] (\varXone -\varLinelengthShort,-\varLinelengthShort/\varSlope + 1) -- (\varXone +\varLinelengthShort, \varLinelengthShort/\varSlope + 1);
                    \draw (\varXone +\varLinelengthShort, \varLinelengthShort/\varSlope + 1) node[above] {$\quad L_{1,\unaryminus 1}$};
                    \draw[-,gray] (\varXone,-\varSlope*\varDist*0.7 + 1)--(\varXone, \varSlope*\varDist*1.2 + 1);
                    \draw (\varXone, \varSlope*\varDist*1.2 + 1) node[above] {$  L_{1,0}$};
                \end{tikzpicture}
            }
            \hfill
         \subfloat[ A configuration of three static particles, located in $(x_1,0)$, $(x_2,0)$ and $(x_3,0)$, with the values $\gamma_{i,k}$ propagating along the lines $L_{i,k}$.]{\label{fig:ghostb}
                \begin{tikzpicture}[scale = 0.78]
                    \draw[->] (-0.3,0) -- (2+2*\varDist,0);
                    \draw (2+2*\varDist,0) node[right] {$x$};
                    \draw [->] (0,-2) -- (0,2);
                    \draw (0,2) node[above] {$v$};
                    \draw (\varXone,0) node[below] {$x_1$} node {$\bullet$};
                    \draw (\varXtwo,0) node[below] {$x_2$} node{$\bullet$};
                    \draw (\varXthree,0) node[below] {$x_3$} node{$\bullet$};
                    \draw[-,gray] (\varXone -\varLinelengthShort/2,\varLinelengthShort/2/\varSlope) -- (\varXone +\varLinelength, -\varLinelength/\varSlope);
                    \draw[-,gray] (\varXone -\varLinelengthShort/2,-\varLinelengthShort/2/\varSlope) -- (\varXone +\varLinelength, \varLinelength/\varSlope);
                    \draw[-,gray] (\varXone,-\varSlope*\varDist*0.5)--(\varXone, \varSlope*\varDist*0.5);
                    \draw[-,gray] (\varXtwo -\varLinelengthShort,\varLinelengthShort/\varSlope) -- (\varXtwo +\varLinelengthShort, -\varLinelengthShort/\varSlope);
                    \draw[-,gray] (\varXtwo -\varLinelengthShort,-\varLinelengthShort/\varSlope) -- (\varXtwo +\varLinelengthShort, \varLinelengthShort/\varSlope);
                    \draw[-,gray] (\varXtwo,-\varSlope*\varDist*1.2)--(\varXtwo, \varSlope*\varDist*1.15);
                    \draw[-,gray] (\varXthree -\varLinelength,\varLinelength/\varSlope) -- (\varXthree +\varLinelengthShort/2, -\varLinelengthShort/2/\varSlope);
                    \draw[-,gray] (\varXthree -\varLinelength,-\varLinelength/\varSlope) -- (\varXthree +\varLinelengthShort/2, \varLinelengthShort/2/\varSlope);
                    \draw[-,gray] (\varXthree,-\varSlope*\varDist*0.5)--(\varXthree, \varSlope*\varDist*0.5);
                    \draw (\varXtwo, \varSlope*\varDist) node[above=0.2] {$L_{1,\unaryminus 1}\cap L_{2,0} \cap L_{3,1} $} node {$\bigcirc$} node[right] {$\ \leftarrow \frac{\gamma_{1,\unaryminus 1}{+}\gamma_{2,0}{+}\gamma_{3,1}}{3}$};
                    \draw (\varXtwo, -\varSlope*\varDist) node {$\bigcirc$};
                    \draw (\varXone,0.7) node {$\gamma_{1,0}$};
                    \draw (\varXone + 0.7, 0.4) node [rotate = 45] {$\gamma_{1,\unaryminus 1}$};
                    \draw (\varXone - 0.7, 0.4) node [rotate = -45] {$\gamma_{1,1}$};
                    \draw (\varXthree,0.7) node {$\gamma_{3,0}$};
                    \draw (\varXthree + 0.7, 0.4) node [rotate = 45] {$\gamma_{3,\unaryminus 1}$};
                    \draw (\varXthree - 0.7, 0.4) node [rotate = -45] {$\gamma_{3,1}$};
                    \draw (\varXtwo, -0.9) node {$\gamma_{2,0}$};
                    \draw (\varXtwo + 0.7, -0.4) node [rotate = -45] {$\gamma_{2,\unaryminus 1}$};
                    \draw (\varXtwo - 0.7, -0.4) node [rotate = 45] {$\gamma_{2,1}$};
                \end{tikzpicture}
            }
        \caption{The geometries of the problem for some simple configurations in one dimension ($d=1$) with three  time measurements ($K=1$) and $\calK=\{-1,0,1\}$}.
        \label{fig:ghost}
    \end{figure}
    A natural way to build dual certificates for the dynamical problem is to consider dual certificates for the static problems at each time step in $\calK$ and then to average them. In particular, we make the choice $\gamma_{i,k}=\eta_i$ for every $i$ and $k$. More precisely, we have the following definition.

    \begin{definition}[Static average certificate]
        \label{def:staticAverage}  Take $T=\{(x_i,v_i)\}_{i=1,\dots N}\subseteq\Omega$. Let $\eta \in \sgn\K^N$  and $\calK \subseteq \{-K, \ldots, K\}$ with $|\calK| \geq 3$. Assume  that for every $k \in \calK$ there exists a static dual certificate, i.e.\  there exists $\tilde q_k(x) =\sum_{\norm{l}_\infty \leq f_c} c_{k,l} e^{i 2\pi l \cdot x}$ such that
        \begin{subequations}\label{eq:average}
            \begin{equation} \label{eq:averagea}
                \left\{ \begin{array}{ll} \tilde q_k(x_i + k \tau v_i) = \eta_i, & \quad   i \in \{1,\ldots,N\}, \\
                |\tilde q_k(y)| < 1, & \quad y \in [0,1]^d \setminus \{(x_i+k\tau v_i)\}_{i=1}^N. \end{array} \right.
            \end{equation}
            We call  the function $q(x,v)$ defined as
            \begin{equation}\label{eq:averageb}
                q(x,v) = \frac{1}{|\calK|} \sum_{k \in \calK} q_k(x,v) = \frac{1}{|\calK |} \sum_{k \in \calK} \tilde q_k(x + k \tau  v)
            \end{equation}
            a \emph{$\calK$-static average certificate}.
        \end{subequations}
    \end{definition}

If static dual certificates exist for every time sample $k\in\calK$, we can immediately build a  static average certificate $q(x,v)$ by using \eqref{eq:averageb}. By construction, it satisfies
    \begin{equation} \label{eq:staavebou}
        \left\{ \begin{array}{ll} q(x_i,v_i) = \eta_i, & \quad  i \in \{1,\ldots,N\}, \\ |q(x,v)| \leq 1, & \quad (x,v)\in\Omega \setminus \T. \end{array} \right. 
    \end{equation}
        This function almost satisfies \eqref{eq:cond}, except that it may happen that $|q(x,v)| = 1$ for some $(x,v)\in\Omega \setminus \T$.
     Take as example the configuration of points given in  \Cref{fig:ghostb} with $\eta_1= \eta_2 = \eta_3 = 1$ and $\calK=\{-1,0,1\}$. The static average certificate will value $1$ in each of the particles and by construction $\gamma_{i,k} = 1$. As a consequence, $q$ will have value $1$ also in the circled points, in which $|\calK|=3$  lines $L_{i,k}$ intersect, hence breaking condition \eqref{eq:cond}.

    In order to ensure that the configuration of particles $\{(x_i,v_i)\}_i$ admits a dual certificate, we characterize these conflictive points.

    \begin{definition}[Ghost particles] Let $\{(x_i, v_i)\}_{i=1,\dots,N}\subseteq\Omega$ be a configuration of particles and  $\calK = \{k_1,\ldots,k_m\}$ be $m = |\calK|$ time samples. A point $(g,w) \in \Omega$ is a \emph{ghost particle} if there exists a set of different indexes $i_1,\ldots, i_m \in \{ 1,\ldots,N \}$ such that
        \begin{equation*}
            \bigcap_{p=1}^{m} L_{i_p,k_p} = \{ (g,w) \}.
        \end{equation*}
    \end{definition}
     To give the intuition behind the definition of ghost particles, recall that the elements of $\Omega$ represent the trajectories of moving objects in $[0,1]^d$: given a time $k$, a particle  $(x,v)\in\Omega$ describes an object located in $x + k \tau v$. From this we notice that the set $L_{i,k}$ represents all possible moving objects in $[0,1]^d$ that at time $k$ would be placed at the same location as the particle $i$, since $x+ \tau k v=x_i+ \tau k v_i$. Therefore, ghost particles can be understood as  possible objects that at every time sample share their location with a given particle.  In the example presented in \Cref{fig:ghostb}, the highlighted ghost point on top of the particles represents an object moving from left to right, that for $k= -1, 0,$ and $1$, would be located in $x_1,x_2$ and $x_3$ respectively.

\subsection{ Main result }
We are now ready to state the main result of this section. Several comments on the assumptions are given after the proof.

    \begin{thm}\label{thm:main} Let $T=\{(x_i,v_i)\}_{i=1,\dots,N}$ be  a configuration of $N$ particles,  $w \in \K^N$ and  $\calK\subseteq\{-K,\dots,K\}$ be such that $|\calK|\ge 3$. Let
        \[
            \omega = \sum_{i=1}^N w_i\delta_{(x_i,v_i)}\in\mathcal{M}(\Omega)
        \]
        be the unknown measure to be recovered. Assume that:
        \begin{enumerate}[(1)]
            \itemsep0em
            \item for every $k\in\calK$ and  $\eta \in \sgn\K^N$ there exists a static dual certificate $\tilde q_k(x) $ satisfying \eqref{eq:averagea};
            \item and the configuration does not admit ghost particles.
        \end{enumerate}	
        Then $\omega$ is the unique solution of the dynamical recovery problem \eqref{eq:lasso_v}, where $\mathcal{G}$ is given by \eqref{eq:GFourier}.
    \end{thm}

    \begin{proof}
        By \Cref{prop:dual}, it is sufficient to construct a dual certificate for the dynamical recovery problem.
        Thanks to assumption (1), for every $\eta\in\sgn\K^N$ we can build a $\calK$-static average certificate $q(x,v)$. By \eqref{eq:staavebou} and assumption (2) it is enough to prove that 
        \begin{equation*}
            |q(x,v)| = 1 \;\implies\; (x,v) \in \T \cup G,
        \end{equation*}
        where $G$ is the set of all the ghost particles of the configuration, since the first condition of \eqref{eq:cond} is automatically satisfied by \eqref{eq:averagea}.

        Note that for $k_1 \neq k_2$ and any two $i,j \in \{1,\ldots,N\}$, the set $L_{i,k_1} \bigcap L_{j,k_2}$ has at most one element. In particular, we    notice that $L_{i,k_1} \bigcap L_{i,k_2} = \{(x_i,v_i)\}$. This observation will be useful below.

        Suppose $|q(x,v)| = 1$. Since $q(x,v)$ is defined as an average of terms $q_k(x,v)$, where each of them satisfies $|q_k(x,v)|\leq 1$, a necessary condition for $|q(x,v)|=1$ is that for every $k\in\calK$ we have $|\tilde q_k(x+k\tau v)|=|q_k(x,v)|=1$. By definition that happens exclusively if for every $k \in \calK$, $(x,v) \in L_{i,k}$ for some $i \in \{ 1,\ldots, N\}$. In other words, there exists a family of indexes $i_k \in \{1,\ldots,N\}$ such that 
         \[
             \{(x,v)\} = \bigcap_{k\in\calK} L_{i_k,k}.
         \]
          There are two cases: if some of these indexes repeat (i.e.\ $i_{k_1} = i_{k_2}$, for $k_1 \neq k_2$), then we know that $(x,v)$ must be equal to particle $i_{k_1}$. In the case none of the indexes $i_k$ repeats, then by definition $(x,v)$ is a ghost particle.
    \end{proof}

    \subsection{Comments on the hypotheses of \Cref{thm:main}}

    Let us now comment on assumptions (1) and (2), and show why these are easily satisfied. Let us start from assumption (1), namely the existence of static dual certificates.
    \begin{rem}\label{rem:thm}
        Take $k\in \{-K,\ldots,K\}$. There exists a static dual certificate  $\tilde q_k(x) $ satisfying \eqref{eq:averagea} in any of the following situations.
        \begin{enumerate}[(a)]
            \itemsep0em 
            \item The particles at time step $k\tau$ are sufficiently separated, namely
                \begin{equation} \label{eq:sep}
                    \norm{(x_i + k \tau v_i) - (x_j + k \tau v_j) }_{\infty} \geq \frac{C_d}{f_c}, \qquad  i\neq j,
                \end{equation}
                and $f_c\ge C'_d$, where $C_d,C'_d>0$ are constants depending only on the dimension\footnote{Several bounds are known for these constants, for instance $C_d=2$ if $d=1$ and $\K=\C$, $C_d=1.87$ if $d=1$ and $\K=\R$, and $C_d=2.38$ if $d=2$ and $\K=\R$. More precise estimates may be derived, which yield slightly better constants.}  \cite{candes2014towards}.
            \item The weights $w_i$ are all positive and the particles at time step $k\tau$ are divided into groups, and within each group a minimum separation condition like \eqref{eq:sep} is satisfied \cite{2016-Morgenshtern-candes} (this holds in the discrete setting).
            \item The weights $w_i$ are all positive, $d=1$ and $f_c\ge 2N$ \cite{de2012exact}. (It is remarkable that in this case no minimum separation condition is required.) \label{rem:positive_weights}
        \end{enumerate}
    \end{rem}
    There are reasons to believe that (c) should be sufficient also if $d>1$ \cite{poon2017multi}, but as far as we are aware a rigorous proof of this fact is still missing. We also would like to remark that assumption (1) is equivalent to being able to solve at least three static problems individually, therefore Theorem~\ref{thm:main} does not completely express how better the fully dynamical problem is (except that it allows for the simultaneous recovery of the velocity and avoids the pre-selection of the proper time samples). In fact, the dynamical reconstruction turns out to perform much better than Theorem~\ref{thm:main} predicts, see Proposition~\ref{prop:removeGhost} and the simulations of Section~\ref{sec:numerical} below.

    Let us now turn to assumption (2), and show that it is satisfied almost surely if the particles $(x_i,v_i)$ are chosen uniformly at random, the following proposition is adapted from \cite{2017radial} to our case.
    \begin{prop}\label{prop:ghost}
        Assume $|\calK|\ge 3$. Let $\{(x_i,v_i)\}^N_{i=1}$ be independent random variables, drawn from absolutely continuous distributions $\mu_i$ supported in $\Omega$. Then almost surely there are no ghost points.
    \end{prop}
    \begin{proof} To simplify the notation denote $P_i = (x_i,v_i)\in \Omega$ and $T_N = \{P_i\}_{i=1}^N$. Set $m=|\calK|$. 
        Let $G(T_n)$ denote the set of ghost particles of a configuration $T_n$ of $n$ particles. We now define the potential ghost particles $\tilde G(T_{n-1})$ of a configuration $T_{n-1}$ of $n-1$ particles. We say that $(g,w)\in \tilde G(T_{n-1})$ if   there exists a particle $P_{n}\in\Omega$ such that $(g,w) \in G(T_{n-1} \cup \{P_{n}\})$, namely
        \[
            \tilde G(T_{n-1}):= \bigcup_{P_n\in\Omega}  G(T_{n-1} \cup \{P_{n}\}).
        \]
        By definition of ghost points, we have that if 
        $(g,w)\in \tilde G(T_{n-1})$ then there exist $m-1$ distinct time samples $k_1,\dots,k_{m-1}\in\calK$	and $m-1$ distinct particles $P_{i_1},\dots,P_{i_{m-1}}\in T_{n-1}$ such that 
        \[
            (g,w)\in \bigcap_{p=1}^{m-1} L_{i_p,k_p}.
        \]
        Notice that  $\tilde G(T_{n-1})$ is always finite. This stems from the fact that for any $k_1\neq k_2$ in $\calK$, the intersection of the sets $L_{i,k_1}, L_{j,k_2}$ are singletons for any $i,j \in \{1,\ldots,n\}$. Since $m\ge 3$ and $T_{n-1}$ and $\calK$ are finite sets, we have  a finite number of sets $L_{i,k}$ to intersect, leading to a finite  number of elements in $\tilde G(T_{n-1})$.

        Now we prove that, for any configuration of particles $T_{n-1}$ such that there are no ghost points, the set of new particles $P_n$ that would generate a ghost point has zero measure. More precisely, if $G(T_{n-1}) = \emptyset$ then
        \begin{equation} \label{eq:zerMea}
            \mu_n \left( \left\{ P_n \in \Omega: G(T_{n-1} \cup \{P_n\}) \neq \emptyset  \right\} \right) = 0.
        \end{equation}
        In order to prove this, notice that if $G(T_{n-1}) = \emptyset$ we have
        \[
            G(T_{n-1} \cup \{P_n\})    \subseteq \{ (g,w) \in \tilde G(T_{n-1}): \exists k \in \calK, \ (g,w) \in L_{n,k} \}.
        \]
        Thus, since $(g,w) \in L_{n,k}$ means $(g-x_n) + \tau k (w-v_n) = 0$ we obtain
        \begin{multline*}
            \left\{ P_n \in \Omega: G(T_{n-1} \cup \{P_n\}) \neq \emptyset \right\}  
            \\
            \subseteq \bigcup_{(g,w) \in \tilde G(T_{n-1})} \ \ \bigcup_{k \in \calK} \{P_n \in \Omega: x_n +\tau k v_n =g + \tau k w  \}.
        \end{multline*}
        Since this is a finite union of affine subspaces of dimension $d$,  it has zero Lebesgue measure. By the absolute continuity of $\mu_n$, we derive \eqref{eq:zerMea}. 

         For $n\in\{2,\dots,N\}$, fix a configuration of  $T_{n-1}$ particles. Denoting $dP_i = dx_i dv_i$, $f_i=\frac{d\mu_i}{dP_i}$ and
        $
         \mathds{1}_S=
         1$ if $S$ is true and 
          $
         \mathds{1}_S=
         0$ if $S$ is false,
        we have
        \begin{align*}
            \int_\Omega \mathds{1}_{G(T_n) \neq \emptyset} f_{n}(P_n) \,dP_n  
            & = \int_\Omega \left( \mathds{1}_{G(T_{n-1}) \neq \emptyset} + \mathds{1}_{G(T_{n-1}) = \emptyset} \mathds{1}_{G(T_{n}) \neq \emptyset}\right)  f_{n}(P_n) dP_n  \\
            & =   \mathds{1}_{G(T_{n-1}) \neq \emptyset} , 
        \end{align*}
        where in the last equality we used \eqref{eq:zerMea} and that  $\int_\Omega f_{n}(P_n) dP_n=1$. Using this property $N-m+1$ times for $n=N,N-1,\dots,m$ and setting $\mu=\otimes_{i=1}^N \mu_{i}$, we obtain
            \begin{align*}
                \mu (\left\{ (P_1,\ldots,P_N) : G(T_N) \neq \emptyset \right\})  &= \int_\Omega \dots \int_\Omega \mathds{1}_{G(T_N) \neq \emptyset} \prod_{i=1}^{N} f_{i}(P_i) dP_1\dots dP_N \\ & = \int_\Omega \dots \int_\Omega \mathds{1}_{G(T_{N-1}) \neq \emptyset} \prod_{i=1}^{N-1} f_{i}(P_i) dP_1\dots dP_{N-1}\\
                &=\cdots\\
                & = \int_\Omega \dots \int_\Omega \mathds{1}_{G(T_{m-1}) \neq \emptyset} \prod_{i=1}^{m-1} f_{i}(P_i) dP_1\dots dP_{m-1} \\
                & = 0,
            \end{align*}
        where the last equality follows from $\mathds{1}_{G(T_{m-1}) \neq \emptyset}=0$, since there cannot be any ghost particles if there are more time samples than particles.
    \end{proof}

    On the other hand, if ghost particles do arise, the conclusion of \Cref{thm:main} may not be true, even if a minimum separation condition is satisfied at all time steps (so that assumption (1) is satisfied). Even though the probability of a random configuration of particles to produce ghost particles is zero, it is worth considering it since  the stability of the problem will deteriorate for nearby configurations (see \Cref{sec:stability}). 

    In \Cref{fig:undetectable} we provide an example of this case, with three time measurements ($K=1$),  three particles $P_1$, $P_2$ and $P_3$ and three ghost particles $G_1$, $G_2$ and $G_3$. The configuration is constructed in such a way that, at each time step, the positions of the ghost particles coincide with those of the physical particles, thereby producing the same measurements. In other words, we have $\mathcal{G}(\sum_i \delta_{P_i})=\mathcal{G}(\sum_i \delta_{G_i})$ and $\norm{\sum_i \delta_{P_i}}_{TV}=\norm{\sum_i \delta_{G_i}}_{TV}$, and so the minimization problem \eqref{eq:lasso_v} has multiple solutions.

    In the following result, we generalize this observation to more general configurations.

    \begin{figure}
        \begin{center}
            \begin{tikzpicture}[scale = 0.9]
                \newcommand{\varExtra}{0.3};
                \draw[-,gray] (-1-\varExtra, 2+\varExtra) -- (1+\varExtra, -\varExtra);
                \draw[-,gray] (-1-\varExtra, \varExtra) -- (\varExtra, -\varExtra-1);
                \draw[-,gray] (-\varExtra, 3+\varExtra) -- (1+\varExtra, 2-\varExtra);
                \draw[-,gray] (-\varExtra, -1 -\varExtra) --( 1+\varExtra, \varExtra);
                \draw[-,gray] (-1-\varExtra,-\varExtra) -- (1+\varExtra, 2+\varExtra);
                \draw[-,gray] (-1-\varExtra, 2-\varExtra) -- (\varExtra, 3+\varExtra);
                \draw[-,gray] (-1,-\varExtra/12) -- (-1,2+\varExtra);
                \draw[-,gray] (0,-1-\varExtra) -- (0, 3+\varExtra);
                \draw[-,gray] (1, -\varExtra/12) -- (1, 2+\varExtra);
                \draw (-1,0) node {$\bullet$} node[below] {\scriptsize $P_1$};
                \draw (1,0) node {$\bullet$} node[below] {\scriptsize $P_2$};
                \draw (0, 3) node {$\bullet$} node[right] {\scriptsize $P_3$};
                \draw (0, -1) node {$\circ$} node[right] {\scriptsize $G_1$};
                \draw (-1, 2) node {$\circ$}node[right] {\scriptsize $G_2$};
                \draw (1,2) node {$\circ$} node[right] {\scriptsize $G_3$};	
                \draw[->] (-1.7,-1) -- (-1.7,3);
                \draw (-1.7,3) node [above] {$v$};
                \draw[->] (-2,0) -- (2,0);
                \draw (2,0) node [right] {$x$};
                \newcommand{\legendLocx}{1.4};
                \newcommand{\legendLocy}{4.8};
                \draw[-] (\legendLocx,\legendLocy - 0.95) -- (\legendLocx + 2.9,\legendLocy - 0.95) -- (\legendLocx + 2.9,\legendLocy) -- (\legendLocx,\legendLocy) -- (\legendLocx,\legendLocy - 0.95);
                \draw (\legendLocx + 0.3, \legendLocy - 0.3) node {$\bullet$} node[right] { \footnotesize \ \ Particle};
                \draw (\legendLocx + 0.3,\legendLocy - 0.7) node {$\circ$} node[right] {\footnotesize \ \ Ghost particle};
                \newcommand{\loLeftx}{4.5};
                \newcommand{\loLefty}{-0.7};
                \newcommand{\verSep}{1.8};
                \newcommand{\horSep}{0.6};
                \newcommand{\horLength}{8*\horSep};
                \newcommand{\verLength}{0.2};
                \newcommand{\verSepPar}{0.25};
                \newcommand{\arrowLen}{0.18};
                \newcommand{\drawArrow}[4]{\draw[->] (#1,#2)--(#1 #3 #4*\arrowLen,#2)};
                \newcommand{\vertLin}[2]{ \draw[-] (#1,#2-\verLength/2)--(#1,#2+\verLength/2)};
                \draw[-] (\loLeftx,\loLefty)--(\loLeftx + \horLength,\loLefty);
                \draw[-] (\loLeftx,\loLefty+\verSep)--(\loLeftx + \horLength,\loLefty+\verSep);
                \draw[-] (\loLeftx,\loLefty+2*\verSep)--(\loLeftx + \horLength,\loLefty+2*\verSep);
                \draw (\loLeftx,\loLefty) node[left] {$k \scalebox{0.7}[1.0]{\( \ = \ \)} 1$ \hspace{0.01cm}};	
                \draw (\loLeftx,\loLefty + \verSep) node[left] {$k \scalebox{0.7}[1.0]{\( \ = \ \)}  0$ \hspace{0.01cm} };		
                \draw (\loLeftx,\loLefty + 2*\verSep) node[left] {$k \scalebox{0.7}[1.0]{\( \ = \ \)}  \unaryminus 1$};	
                \newcounter{forfor}
                \newcounter{forHeight}
                \forloop{forHeight}{0}{\value{forHeight}<3}{
                    \forloop{forfor}{1}{\value{forfor}<8}{
                        \vertLin{\loLeftx+\horSep*\value{forfor}}{\loLefty + \verSep*\value{forHeight}}; 
                } }
                \draw (\loLeftx+\horSep,\loLefty + 2*\verSep + \verSepPar) node {$\bullet$} node[above] {\scriptsize $P_3$};
                \drawArrow{\loLeftx+\horSep}{\loLefty + 2*\verSep + \verSepPar}{+}{3};
                \draw (\loLeftx+3*\horSep,\loLefty + 2*\verSep + \verSepPar) node {$\bullet$} node[above] {\scriptsize $P_1$};
                \draw (\loLeftx+5*\horSep,\loLefty + 2*\verSep + \verSepPar) node {$\bullet$} node[above] {\scriptsize $P_2$};
                \draw (\loLeftx+\horSep,\loLefty + 2*\verSep - \verSepPar) node {$\circ$} node[below] {\scriptsize $G_2$};
                \drawArrow{\loLeftx+\horSep}{\loLefty + 2*\verSep - \verSepPar}{+}{2};
                \draw (\loLeftx+3*\horSep,\loLefty + 2*\verSep - \verSepPar) node {$\circ$} node[below] {\scriptsize $G_3$};
                \drawArrow{\loLeftx+3*\horSep}{\loLefty + 2*\verSep - \verSepPar}{+}{2};
            \draw (\loLeftx+5*\horSep,\loLefty + 2*\verSep - \verSepPar) node {$\circ$} node[below] {\scriptsize $G_1$};
            \drawArrow{\loLeftx+5*\horSep}{\loLefty + 2*\verSep - \verSepPar}{-}{1};
            \draw (\loLeftx+4*\horSep,\loLefty + \verSep + \verSepPar) node {$\bullet$} node[above] {\scriptsize $P_3$};
            \drawArrow{\loLeftx+4*\horSep}{\loLefty + \verSep + \verSepPar}{+}{3};
            \draw (\loLeftx+3*\horSep,\loLefty + \verSep + \verSepPar) node {$\bullet$} node[above] {\scriptsize $P_1$};
            \draw (\loLeftx+5*\horSep,\loLefty + \verSep + \verSepPar) node {$\bullet$} node[above] {\scriptsize $P_2$};
            \draw (\loLeftx+4*\horSep,\loLefty + \verSep - \verSepPar) node {$\circ$} node[below] {\scriptsize $G_1$};
            \drawArrow{\loLeftx+4*\horSep}{\loLefty + \verSep - \verSepPar}{-}{1};
            \draw (\loLeftx+3*\horSep,\loLefty + \verSep - \verSepPar) node {$\circ$} node[below] {\scriptsize $G_2$};
            \drawArrow{\loLeftx+3*\horSep}{\loLefty + \verSep - \verSepPar}{+}{2};
            \draw (\loLeftx+5*\horSep,\loLefty + \verSep - \verSepPar) node {$\circ$} node[below] {\scriptsize $G_3$};
            \drawArrow{\loLeftx+5*\horSep}{\loLefty + \verSep - \verSepPar}{+}{2};
            \draw (\loLeftx+7*\horSep,\loLefty  + \verSepPar) node {$\bullet$} node[above] {\scriptsize $P_3$};
            \drawArrow{\loLeftx+7*\horSep}{\loLefty + \verSepPar}{+}{3};
            \draw (\loLeftx+3*\horSep,\loLefty  + \verSepPar) node {$\bullet$} node[above] {\scriptsize $P_1$};
            \draw (\loLeftx+5*\horSep,\loLefty  + \verSepPar) node {$\bullet$} node[above] {\scriptsize $P_2$};
            \draw (\loLeftx+7*\horSep,\loLefty  - \verSepPar) node {$\circ$} node[below] {\scriptsize $G_3$};
            \drawArrow{\loLeftx+7*\horSep}{\loLefty - \verSepPar}{+}{2};
            \draw (\loLeftx+3*\horSep,\loLefty  - \verSepPar) node {$\circ$} node[below] {\scriptsize $G_1$};
            \drawArrow{\loLeftx+3*\horSep}{\loLefty - \verSepPar}{-}{1};
            \draw (\loLeftx+5*\horSep,\loLefty  - \verSepPar) node {$\circ$} node[below] {\scriptsize $G_2$};
            \drawArrow{\loLeftx+5*\horSep}{\loLefty - \verSepPar}{+}{2};
        \end{tikzpicture}
    \end{center}
        \caption{In the case of three time measurements, a configuration of points and speeds that allow multiple reconstructions. The lines in the left diagram represent $L_{i,k}$ for each time sample $k$ and particle $i$. On the right hand side we can observe the relative position of each particle at each time step. }
        \label{fig:undetectable}
\end{figure}
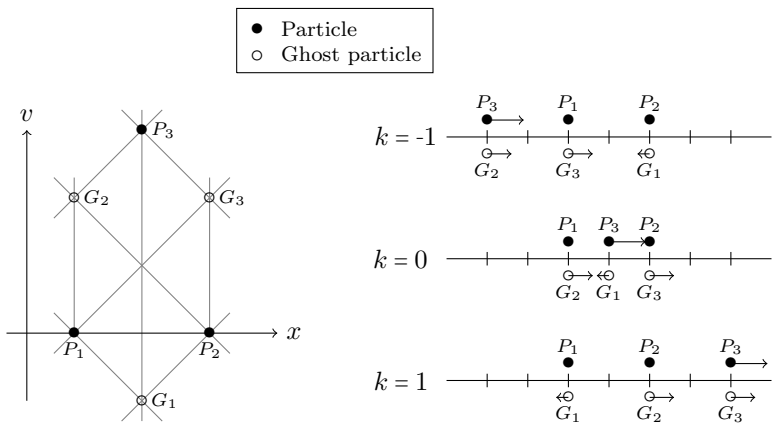

    \begin{prop} \label{prop:undetect}
        Take $w\in\R_+^{m}$, with $m = \lvert \calK \rvert\ge 3$, and let $\{(x_i,v_i)\}_{i=1}^m \subseteq \Omega$ be a configuration of $m$ distinct particles admitting  $m$ distinct ghost particles $\{(g_j, w_j)\}_{j=1}^m \subseteq \Omega$. Let $\omega=\sum_{i=1}^{m}w_i\delta_{(x_i,v_i)}$. Suppose that for every $k \in \calK $ and every $i \in \{1,\ldots,m\}$ there exists a unique ghost particle $(g_j,w_j)$ in the affine space $L_{i,k}$, i.e.
        \begin{equation*}
            g_j-x_i + k \tau(w_j - v_i) = 0.
        \end{equation*}
        Then the minimization problem in \eqref{eq:lasso_v} admits infinitely many  solutions. 
    \end{prop}
    \begin{proof}
        Consider the measure 
        \begin{equation*}
            h = \sum_{i=1}^{m} \delta_{(x_i,v_i)} - \sum_{j=1}^{m} \delta_{(g_j,w_j)}\in \calM (\Omega).
        \end{equation*}
        First notice that $h \neq 0$. If we had $h = 0$,  each particle  would be a ghost particle. 
        By definition of ghost particle, we would have  $(g,w)\in L_{i_1,k_1}\cap L_{i_2,k_2}$ for some $k_1\neq k_2$ in $\calK$, and so
        \[
            (g,w),(x_{i_1},v_{i_1})\in L_{i_1,k_1},\qquad
        (g,w),(x_{i_2},v_{i_2})\in L_{i_2,k_2}.
    \]
        Since $(x_{i_1},v_{i_1})$ and $(x_{i_2},v_{i_2})$ are two different ghost particles, either  $L_{i_1,k_1}$ or $L_{i_2,k_2}$ would contain two ghost particles, contradicting the hypotheses.

        The measure $h$ is undetectable, since it belongs to the kernel of the operator  $\mathcal{G}$ defined in \eqref{eq:G}, as we now show. We readily compute
        \[
            \left(\mathcal{G}h\right)_{l,k}  =  \left< h , \varphi_{l,k} \right>    = \sum_{i = 1}^{m} \varphi_l(x_i + k \tau v_i) - \sum_{j = 1}^{m} \varphi_l(g_j + k \tau \eta_j) =0.
        \]
        Indeed, by our hypothesis on the ghost particles, we have that at every time sample $k$, for every position $x_i + k \tau v_i \in \Omega$ there exists only one ghost point such that $g_j + k \tau w_j = x_i + k \tau v_i$. Therefore, each term of the first sum cancels out with one term of the second sum, as desired.

         For $\beta \in [0, \min_i w_i]$, consider the measure
        \begin{equation*}
            \omega_\beta = \omega - \beta h = \sum_{i = 1}^{m} ( w_i - \beta) \delta_{(x_i,v_i)} +\beta \sum_{j = 1}^{m} \delta_{(g_j, w_j)}.
        \end{equation*}
         Since $\mathcal{G}\omega_\beta = \mathcal{G}\omega$ and $\norm{\omega_\beta}_{TV} = \sum_i ( w_i-\beta) + \sum_j \beta = \sum_i  w_i = \norm{\omega}_{TV}$, we obtain that $\omega_\beta$ is a solution to \eqref{eq:lasso_v} for every $\beta \in [0, \min_i w_i]$.
    \end{proof}

    It is worth observing that the non-uniqueness of  solutions arises also if there exists a subset of the particles satisfying the conditions of \Cref{prop:undetect}.

\section{Other constructions of dynamical dual certificates}\label{sec:other}

In this section we show that the construction of dynamical dual certificates as static average certificates, although natural and efficient, is not the end of the story. In other words, exact recovery may be guaranteed even if assumptions (1) and (2) of \Cref{thm:main} are not satisfied. 
More precisely, we provide alternative constructions of dynamical dual certificates for configurations that either do not allow static dual certificates (assumption (1)) or have ghost particles (assumption (2)). In particular, the first case shows an advantage of our space-velocity model over applying static reconstructions at each time sample. This aspect will also be investigated in \Cref{sec:numerical} below.

\subsection{Dual certificates with no static separation condition}

    The following example of dual certificate is purely numerical, but  shows the possibility of constructing a dual certificate in cases in which static reconstructions are expected to fail.

    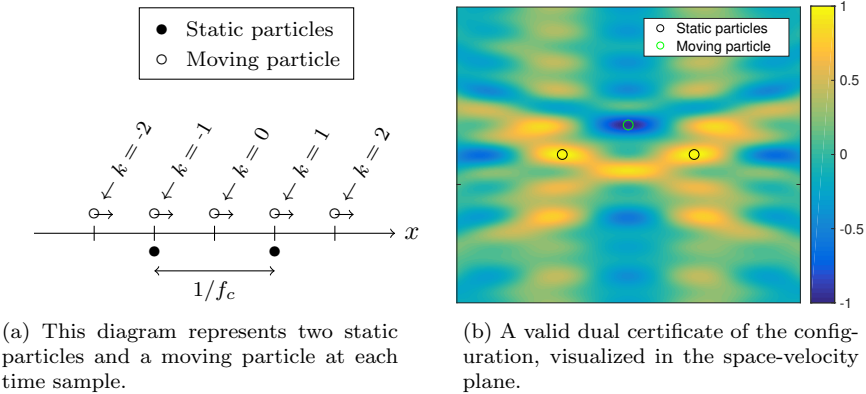
\begin{figure}
        \centering{}\captionsetup[subfigure]{width=150pt}
        \subfloat[This  diagram represents two static particles and a moving particle at each time sample.]{\label{fig:2static1movinga}
        \centering
        \newcommand{\varDist}{2}
        \newcommand{\varSlope}{1}
        \newcommand{\varLinelength}{2.3}
        \newcommand{\varLinelengthShort}{1.3}
        \newcommand{\varXone}{1}
        \newcommand{\varXtwo}{\varXone+\varDist}
        \newcommand{\varXthree}{\varXone + 2*\varDist}
                \begin{tikzpicture}
                    \newcommand{\coordy}{0};
                    \newcommand{\coordx}{0};
                    \newcommand{\horSep}{0.8};
                    \newcommand{\horLength}{6*\horSep};
                    \newcommand{\verLength}{0.2};
                    \newcommand{\verSepPar}{0.25};
                    \newcommand{\arrowLen}{0.25};
                    \newcommand{\drawArrow}[3]{\draw[->] (#1,#2)--(#1 #3 \arrowLen,#2)};
                    \newcommand{\vertLin}[2]{ \draw[-] (#1,#2-\verLength/2)--(#1,#2+\verLength/2)};

                    \draw[->] (\coordx,\coordy)--(\coordx + \horLength,\coordy);
                    \newcounter{forfors}
                    \forloop{forfors}{1}{\value{forfors}<6}{
                        \vertLin{\coordx+\horSep*\value{forfors}}{\coordy};
                    }
                    \draw (\coordx + \horLength,\coordy) node[right] {$x$};
                    \draw (\coordx+2*\horSep, \coordy - \verSepPar) node {$\bullet$};
                    \draw (\coordx+4*\horSep, \coordy - \verSepPar) node {$\bullet$};
                    \draw[<->] (\coordx+2*\horSep, \coordy - 2*\verSepPar)--(\coordx+4*\horSep, \coordy - 2*\verSepPar);
                    \draw (\coordx+3*\horSep, \coordy - 2*\verSepPar) node[below] {\footnotesize $1/f_c$};
                    \draw (\coordx+\horSep*1,\coordy + \verSepPar) node {$\circ$} node[right, rotate=60] { \footnotesize \  $\leftarrow k=\unaryminus 2$};
                    \drawArrow{\coordx+\horSep*1}{\coordy + \verSepPar}{+};
                    \draw (\coordx+\horSep*2,\coordy + \verSepPar) node {$\circ$} node[right, rotate=60] { \footnotesize \  $\leftarrow k=\unaryminus 1$};
                    \drawArrow{\coordx+\horSep*2}{\coordy + \verSepPar}{+};
                    \draw (\coordx+\horSep*3,\coordy + \verSepPar) node {$\circ$} node[right, rotate=60] { \footnotesize \  $\leftarrow k=0$};
                    \drawArrow{\coordx+\horSep*3}{\coordy + \verSepPar}{+};
                    \draw (\coordx+\horSep*4,\coordy + \verSepPar) node {$\circ$} node[right, rotate=60] { \footnotesize \  $\leftarrow k=1$};
                    \drawArrow{\coordx+\horSep*4}{\coordy + \verSepPar}{+};
                    \draw (\coordx+\horSep*5,\coordy + \verSepPar) node {$\circ$} node[right, rotate=60] { \footnotesize \  $\leftarrow k=2$};
                    \drawArrow{\coordx+\horSep*5	}{\coordy + \verSepPar}{+};
                    \newcommand{\legendLocx}{1.4	};
                    \newcommand{\legendLocy}{3.0};
                    \draw[-] (\legendLocx,\legendLocy - 0.95) -- (\legendLocx + 2.85,\legendLocy - 0.95) -- (\legendLocx + 2.85,\legendLocy) -- (\legendLocx,\legendLocy) -- (\legendLocx,\legendLocy - 0.95);
                    \draw (\legendLocx + 0.3, \legendLocy - 0.3) node {$\bullet$} node[right] { \footnotesize \ \ Static particles};
                    \draw (\legendLocx + 0.3,\legendLocy - 0.7) node {$\circ$} node[right] {\footnotesize \ \ Moving particle};
                \end{tikzpicture}
        }
        \hfill{} \subfloat[A valid dual certificate of the configuration, visualized in the space-velocity plane.]{\label{fig:2static1movingb}
        \centering
                \includegraphics[width=0.45\textwidth]{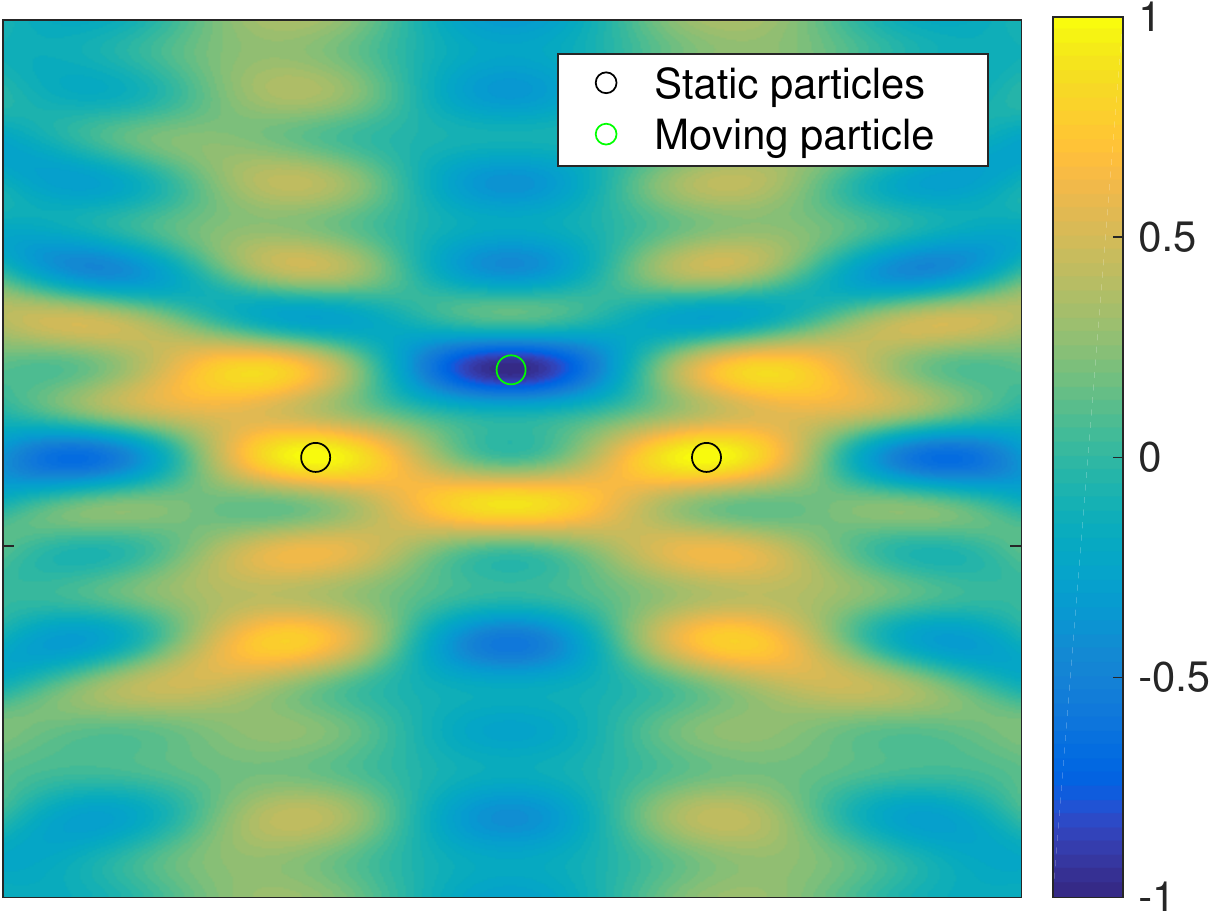}		
        }
        \caption{An example of a  dynamical dual certificate whithout a static minimum separation condition. }
        \label{fig:2static1moving}
    \end{figure}

    The chosen configuration is presented in  \Cref{fig:2static1movinga}, where we consider the one-dimensional case ($d=1$), five time measurements ($K=2$), two static particles barely separated enough to allow a reconstruction and a third moving particle. We give positive weights $ w_i$ to the static particles, and a negative weight to the  moving one.
     Since this third particle is at each time measurement close to another particle and has a different sign, it is not possible to localize it with a static reconstruction at any point. We recall that $f_c$ is the maximum imaging frequency, and $\frac{1}{2f_c}$ is far below the optimal minimum separation distance (see \Cref{rem:thm} and  \cite[Section~5]{candes2014towards}). 

    In  \Cref{fig:2static1movingb} we can see a dual certificate for this configuration (with  $f_c=20$). To obtain this dual certificate, we followed the construction done in \cite{candes2014towards} for the two dimensional case, but considering only the frequencies available in our setting given by the set \eqref{eq:butterflySet}.

    \subsection{ Dual certificates in presence of ghost particles }

As we saw in \Cref{fig:ghostb}, in the presence of ghost points the static average dual certificate constructed in \eqref{eq:average} is not a valid dynamical dual certificate if the values $\eta_i$ have  a constant sign. Indeed,  the static average dual certificate will have absolute value equal to $1$ in the ghost particles, since the values $\gamma_{i,k}$ on $L_{i,k}$ are simply set to $\eta_i$ so that $\frac{1}{|\calK|}\sum_{k\in\calK} \gamma_{i,k}=\eta_i$ (see \Cref{fig:staticCertificatePerturba}). However, by making a slightly different choice for $\gamma_{i,k}$ it is possible to 
have  simultaneously $q(x_i,v_i)=\eta_i$ and $|q(g,w)|<1$ on every ghost particle, thereby obtaining a valid dual certificate (see \Cref{fig:staticCertificatePerturbb}).

    \begin{figure}
        \centering{}\captionsetup[subfigure]{width=150pt}
        \subfloat[The static average dual certificate (corresponding to $\epsilon=0$) is not a valid dual certificate since it has value $1$ in the two ghost particles.]{\label{fig:staticCertificatePerturba}
            \includegraphics[width=0.45\textwidth]{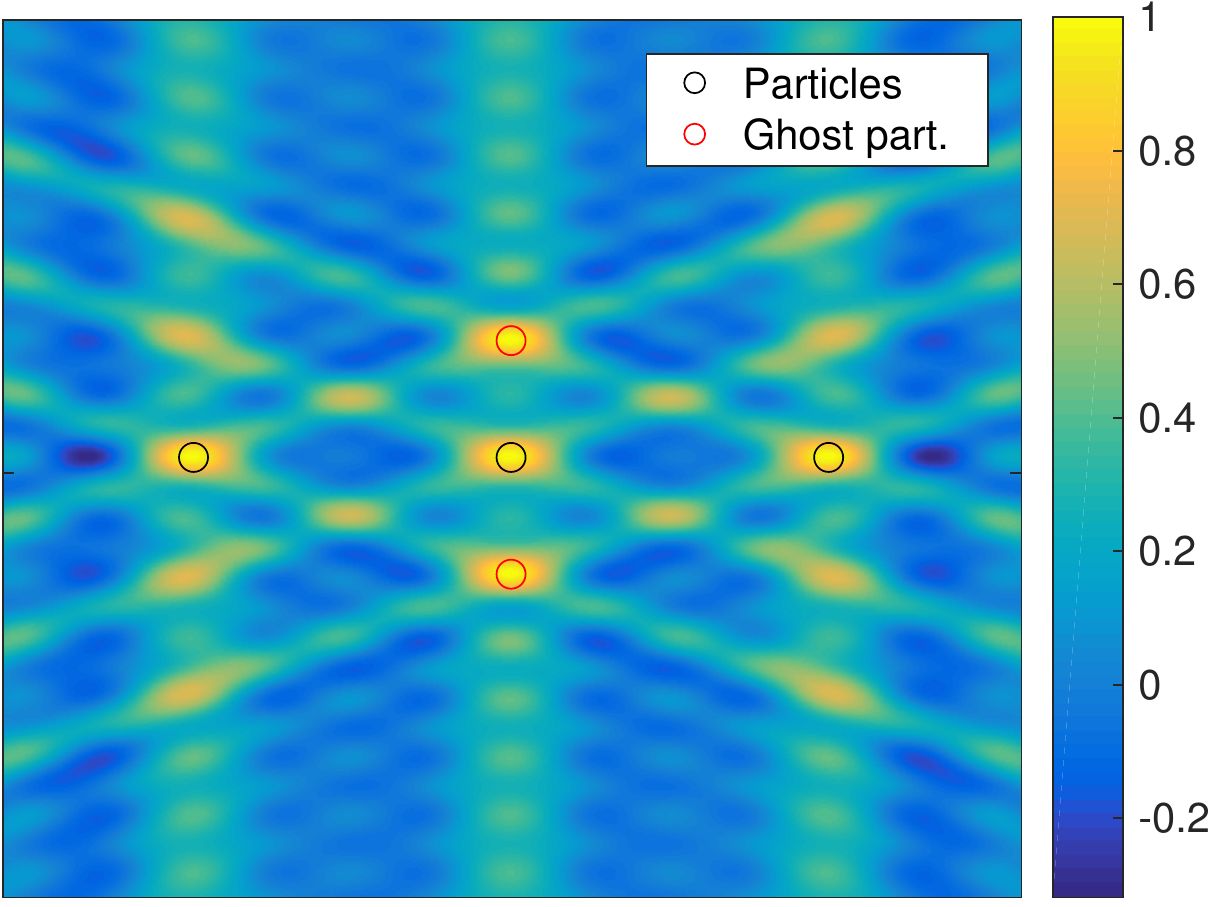}
        }
        \hfill{}
         \subfloat[The valid dual certificate constructed as a perturbation of the static average dual certificate by taking $\epsilon = 0.08$.]{\label{fig:staticCertificatePerturbb}
            \includegraphics[width=0.45\textwidth]{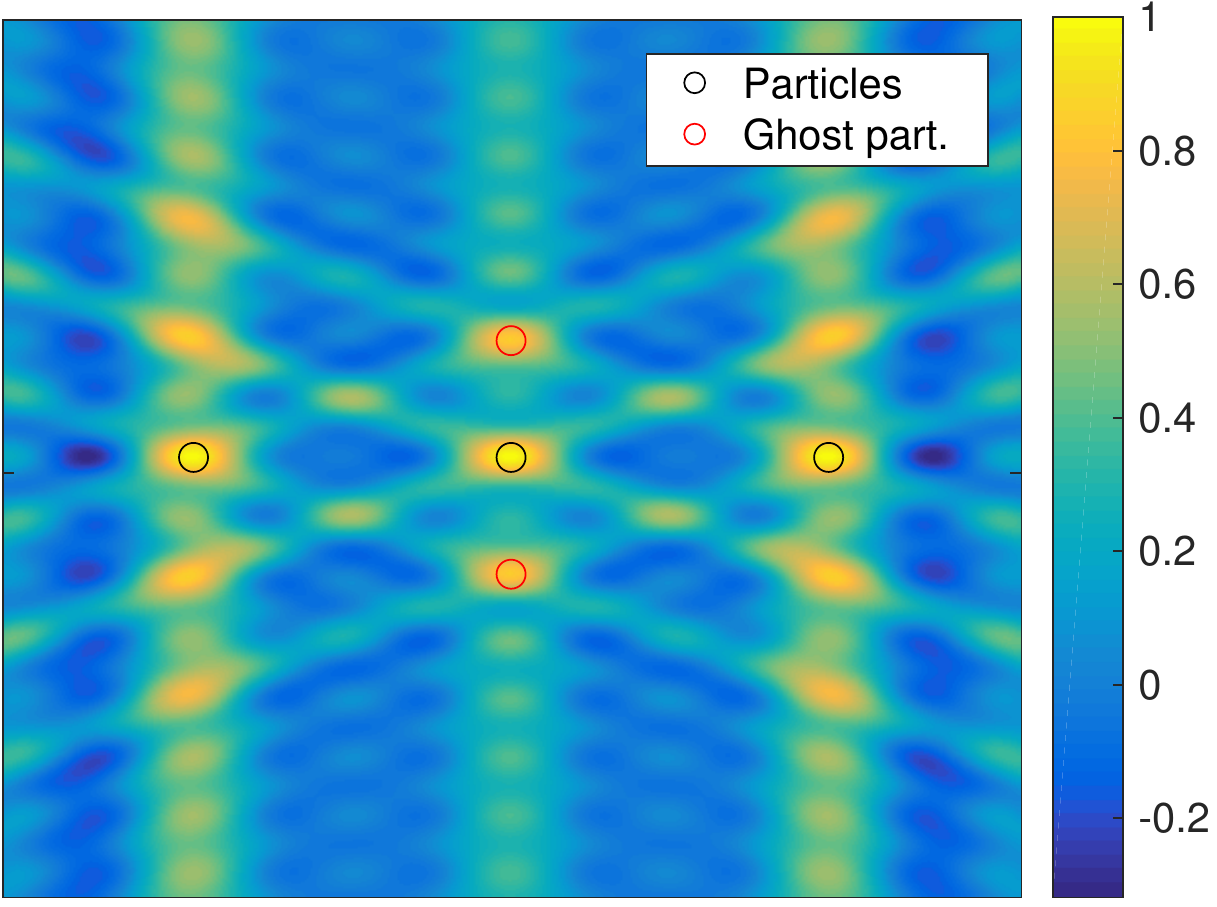}
        }
        \caption{Dual certificates for the configuration of particles of \Cref{fig:ghostb}.
        }
        \label{fig:staticCertificatePerturb}
    \end{figure}

    This construction is formalized in the following result.

    \begin{prop} \label{prop:removeGhost}
        Take $d=1$, $f_c\ge 128$, $K=1$ and $\Delta x\in [\frac{1.87}{f_c},1)$. Consider 3 particles with locations $(-\Delta x,0)$, $(0,0)$ and $ (\Delta x,0)$. Then for every $\eta\in\sgn\K^3$ there exists a dual certificate.
    \end{prop}
    \begin{proof}
        If $\eta$ is not a constant vector, the static average certificate (whose existence is guaranteed by \Cref{rem:thm}a) given by  \eqref{eq:average} is a valid dual certificate, and the result is trivial. Without loss of generality, suppose now that $\eta=(1,1,1)$. Set $x_{\unaryminus 1}= -\Delta x$, $x_0=0$ and $x_1=\Delta x$.

        We construct the dual certificate $q(x,v)$ using the notation of $\S$\ref{sub:dual} with $\calK=\{-1,0,1\}$. More precisely, we write
        \[
            q(x,v)=\frac13\sum_{k=-1}^1 q_k(x,v) 
        = \frac13\sum_{k=-1}^1 \tilde q_k(x+k\tau v),
    \]
        where the functions $\tilde q_k$ are constructed as follows.
        The quantities $\gamma_{i,k}$ defined in \eqref{eq:gamma} need to satisfy \eqref{prop:gammas}, but also to keep the absolute value of $q$ below 1 in the ghost particles. In view of the geometrical configuration (see \Cref{fig:ghostb}), these conditions are:
        \begin{align*}
            \gamma_{i,\unaryminus 1} + \gamma_{i,0} + \gamma_{i,1} & = 3, \qquad i \in \{1,2,3\}, \\
            |\gamma_{1,1} + \gamma_{2,0} + \gamma_{3,\unaryminus 1}| & < 3, \\
            | \gamma_{1,\unaryminus 1} + \gamma_{2,0} + \gamma_{3,1}| & < 3.
        \end{align*}	
        To take a family of solutions, for $\epsilon\in(0,1)$ set
        \begin{align*}
            \gamma_{i,\unaryminus 1} = \gamma_{i,1} & = 1 - \epsilon, \quad i \in \{1,3\}, \\  \gamma_{i,0} & = 1 + 2\epsilon, \quad i \in \{1,3\}, \\
            \gamma_{2,\unaryminus 1} = \gamma_{2,0} = \gamma_{2,1} & = 1.
        \end{align*}
        For each $k\in\calK$, let $\tilde q_k(x)=\sum_{l=-f_c}^{f_c} c_{k,l}e^{2\pi i l x}$ be the low frequency trigonometric polynomial such that
        \[
            \tilde q_k(x_i) = \gamma_{i,k},\quad {\tilde q_k}'(x_i)=0,\qquad i=-1,0,1,
        \]
        constructed in \cite{candes2014towards} by solving a linear system of equations, that is invertible when $\Delta x\ge \frac{1.87}{f_c}$, with $\gamma_{i,k}$ appearing in the right hand side.

        The final step is to ensure that $q(x,v) = \sum q_k(x,v)$ satisfies $|q(x,v)|<1$ for every $(x,v)\in \Omega$ that is not a particle.  When $\epsilon = 0$, $\tilde q_k$ is strictly concave in $x_i$ for every $i$ \cite{candes2014towards}. Furthermore, since the quantities $\gamma_{i,k}$ are affine in $\epsilon$, the map $\epsilon\mapsto \tilde q_k$ is affine as a function from $\R$ to $C^\infty(\Omega)$, hence every derivative of $\tilde q_k$ is continuous in $\epsilon$. Therefore, since $\tilde q_k'(x_i)=0$, the local concavity of $\tilde q_k$  is preserved for $\epsilon$ small and, consequently, the local maxima of $\tilde q_k$ are the interpolation points $x_i$. This is enough to prove that $q(x,v)$ attains local  maxima in the particles $(x_i,0)$ and in the ghost particles $(0,\Delta x)$ and $(0,-\Delta x)$. Finally, by continuity of $q(x,v)$ with respect to $\epsilon$, there exists an $\epsilon$ sufficiently small such that $|q(x,v)| < 1$ for every element in $\Omega$ that is not a particle.
    \end{proof}

    The methodology presented in the proof of \Cref{prop:removeGhost} can be applied to more general examples and it can be iterated to deal with the ghost particles one by one. The only requirement is to have a direction in which there are no ghost particles, so it can be used as a ``sink of mass''. In the  example considered, these directions are described by $L_{1,0}$, $L_{3,0}$, $L_{2,\unaryminus 1}$ and $L_{2,1}$.

    \section{ Stable reconstruction with noise }\label{sec:stability}
    Following the analysis done in \cite{candes2014towards} for the static case, we present a stability result for the dynamical problem. We will review their setting and adapt it to our case under the same hypotheses: one dimensional case ($d=1$), discrete setting and a specific type of noise model with bounded  total variation norm. Even though more recent and finer methods would allow for a more general investigation and in the continuous setting  \cite{fernandez2013support,candes2013,spike-acha-2015}, we decided to focus on this simplified case to highlight the role of ghost particles, which affect the stability of the problem, albeit they almost surely do not affect uniqueness in the noiseless case (Proposition~\ref{prop:ghost}). Hence, we do not aim at providing a complete and thorough analysis, but only at underlying this aspect.

Let us consider a discrete grid $\OmGrid$ with space width $\Delta_x>0$ and velocity width $\Delta_v > 0 $, namely
    \begin{equation*}
        \OmGrid = \left( \left(  \Delta_x \cdot \Z \right) \times \left( \Delta_v \cdot \Z \right) \right) \cap \Omega,
    \end{equation*}
    and we will assume throughout this subsection that our particles are located on the grid, i.e.\ $\T \subseteq \OmGrid$ and $\omega \in \calM\left( \OmGrid \right)$. Since we are dealing with a discrete space, we will denote the total variation norm by the $\ell^1$ norm in $\mathbb{K}^{|\Omega_\#|}=\calM\left( \OmGrid \right)$, as they are equivalent. We also recall  from \cite{candes2014towards} the super-resolution factor in space:
    \begin{equation*} 
        SRF_x = \frac{1}{\Delta_x f_c},
    \end{equation*} 
    which can be understood as the ratio between the desired space resolution and the permitted resolution given by the diffraction limit.

    We consider, instead of  \eqref{eq:signalModel}, the following input noise model: \begin{equation*}
        y = \mathcal{G} (\omega + z), \quad \lVert P_{\mathcal{G}} z \rVert_{1} \leq \delta,
    \end{equation*}
    where $z \in \calM\left(\OmGrid\right)$ and $P_{\mathcal{G}} = \Delta_x \mathcal{G}^* \mathcal{G}$ stands for the projection over the frequencies described by \eqref{eq:butterflySet}. The reason of this projection is that all the other frequencies are filtered out by the measurement process of $\mathcal{G}$. We consider a relaxed version of the noiseless problem \eqref{eq:lasso_v}
    \begin{equation} \label{eq:minnoise}
        \min_{\lambda \in \mathcal{M}\left( \OmGrid \right)}\norm{\lambda}_{1} \quad \text{subject to } \;\lVert \mathcal{G}\lambda - y \rVert_1 \leq \frac{ \delta }{\Delta_x \lVert \mathcal{G}^* \rVert_{\ell^1 \to \mathcal{M}}} ,
    \end{equation}
    where $\lVert \mathcal{G}^* \rVert_{\ell^1 \to \mathcal{M}}$ is the operator norm of $\mathcal{G}^*$.

By strengthening the hypotheses of \Cref{thm:main}, we obtain the following stability result.

    \begin{thm} Let  $f_c\ge 128$, $d=1$ and $\Delta_x,\Delta_v>0$ be such that
        \begin{equation}\label{eq:relation}
            \Delta_x^2 \le \frac{K(K+1)}{3}\tau^2 \Delta_v^2.
        \end{equation}
        Let $T=\{(x_i,v_i)\}_{i=1,\dots,N}\subseteq\Omega_\#$ be  a configuration of $N$ particles,  $w \in \K^N$ and  $\calK=\{-K,\dots,K\}$. Let
        $
        \omega = \sum_{i=1}^N  w_i\delta_{(x_i,v_i)}\in\mathcal{M}(\Omega)
        $
        be the unknown measure to be recovered.

        Suppose that
         \begin{enumerate}[(i)]
             \item  for every $k\in\calK$, the minimum separation condition \eqref{eq:sep} holds;
             \item and  for all $(x,v)\in\Omega \setminus \bigcup_{i=1}^N B_i$ we have
                 \begin{equation} \label{def:dynamicalStability}
                     \frac{1}{2K+1} \sum_{k=-K}^{K} \min \Bigl\{\min_{i \in \{1,\ldots,N\} } |x_i - x + k \tau (v_i - v) |^2,\, \frac{0.1649^2}{f_c^2}\Bigr\} \geq \Delta_x^2,
                 \end{equation}
                 where $B_i = \left\{(x,v) \in \Omega: |x-x_i| + K \tau | v - v_i | < 0.1649/f_c  \right\}$ are neighborhoods of each particle in $\T$.	
         \end{enumerate}
         Let $\hat\omega$ be a minimizer of \eqref{eq:minnoise}. Then we have the following stability bound for the error:
        \begin{equation*}
            \norm{\hat \omega  - \omega}_1 \leq C (\textnormal{SRF}_x)^2 \delta
        \end{equation*} 
        for some absolute constant $C>0$.
    \end{thm}
        \begin{rem}
            The stability condition \eqref{def:dynamicalStability} can be understood as an extension of a condition to prevent ghost particles. As we can notice, if $(x,v)$ is a ghost particle, the sum in the left-hand side values 0. 
        \end{rem}

    \begin{proof}
        Take $\eta \in \K^N$ with $|\eta_i| = 1$. Let  $q(x,v)$ be the static average dual certificate given in \eqref{eq:average}, where the each static dual certificate $\tilde q_k$ is constructed as in  \cite{candes2014towards} for every $k\in\calK$, thanks to assumption (i). 
        With an abuse of notation, set $q = q_{i,j} = q\left(i \Delta_x , j \Delta_v \right)$. Let $P_T$ denote the projection onto the space of vectors supported on $T$, namely, $(P_T q)_{i,j} = q_{i,j}$ if $(i \Delta_x, j \Delta_v) \in T$ and $0$ otherwise.

        In view of \cite[Theorem~1.5]{candes2014towards}, using  the same proof we have the following stability bound for the error:
        \begin{equation} \label{eq:stabilityBound}
            \norm{\hat \omega - \omega}_1 \leq \frac{4 \delta}{1-\norm{P_{\T^c} q }_{\infty}}.
        \end{equation}
        It remains to bound  $\norm{P_{\T^c} q }_{\infty}$ in our discrete grid, 	
        namely the values of $q_{i,j}$ outside the set of particles $\T$. In order to do so, we shall use the following estimates from \cite[Lemma~2.5]{candes2014towards}:
        \begin{equation} \label{eq:lemma2-5}
            |\tilde q_k(t)| \leq 1 - C_1 f_c^2(t-(x_i + k \tau v_i))^2, \quad \textnormal{ when } \quad |t - (x_i+k \tau v_i)| \leq C_2/f_c,
        \end{equation}
        and 
        \begin{equation}\label{eq:lemma2-5b}
            |\tilde q_k(t) | \leq 1 - C_1 C_2^2, \quad \text{ when } \min_{(x_i,v_i)\in\T}|t-(x_i+k \tau v_i)|> C_2/f_c,
        \end{equation}
        where $C_1 = 0.3353$ and $C_2 = 0.1649$.

        Take $(\xGrid, \vGrid) \in \OmGrid \setminus T$, we want to bound the maximum of $\left\lvert q(\xGrid,\vGrid)\right\rvert$. We deal with two cases: when $(\xGrid,\vGrid)$ is close to some particle, i.e.\ $(\xGrid,\vGrid) \in B_i$ for some $i$, or when it is not.

        Fix $i$ such that $(\xGrid,\vGrid) \in B_i$, then we can write $(\xGrid,\vGrid) = (x_i + n_x \Delta_x,v_i + n_v \Delta_v)$ with $(n_x,n_v) \in \Z\times\Z \setminus \{(0,0)\}$.  From the definition of $B_i$, we have that we can use estimate \eqref{eq:lemma2-5} for every $k \in \calK$. Therefore
        \begin{align*}
            |q(\xGrid, \vGrid) | & = \frac{1}{2K+1}\left\lvert \sum_{k=-K}^K \tilde q_k\big( x_i + k \tau v_i + (n_x\Delta_x + k \tau n_v \Delta_v) \big)\right\rvert \\
            & \leq \frac{1}{2K+1} \sum_{k=-K}^K \left( 1 - C_1 f_c^2(n_x \Delta_x + k \tau n_v \Delta_v)^2 \right) \\
            & = \frac{1}{2K+1}  \sum_{k=-K}^K \left( 1 - C_1 f_c^2 \left( n_x^2 \Delta_x^2 + 2k \tau n_x n_v \Delta_x \Delta_v  + k^2 \tau^2 n_v^2 \Delta_v^2 \right) \right) \\
            & = 1 - \frac{C_1 f_c^2}{2K+1} \left(\left( (2K+1) n_x^2\Delta_x^2  + \frac{K (K+1)(2K+1)}{3} \tau^2 n_v^2\Delta_v^2 \right) \right) \\
            & \le 1 - C_1 f_c^2 \Delta_x^2 (n_x^2 + n_v^2)\\
            & \le 1 - C_1 f_c^2 \Delta_x^2 , 
        \end{align*}
        where we used \eqref{eq:relation} and that $n_x^2 + n_v^2\ge 1$.

        In the case where $(\xGrid,\vGrid) \not \in B_i$ for every $i$, we take for each $k\in \calK, i_k = \argmin_{i \in \{1,\ldots,N\}} |x_i -\xGrid + k\tau( v_i-\vGrid)|$ and $\calI_k = x_{i_k} - \xGrid + k \tau (v_{i_k} - \vGrid)$. Hence, by \eqref{eq:lemma2-5b} we have
        \begin{align*}
            |q(\xGrid&, \vGrid)|  = \frac{1}{2K+1} \left| \sum_{k = -K}^K \tilde q_k( x_{i_k} + k \tau v_{i_k} - \calI_k)\right| \\
            & \leq 1 - \frac{C_1 f_c^2}{2K+1} \sum_{k = -K}^K \left(\calI_k^2\, \mathds{1}_{(-\infty,0]}\Bigl(|\calI_k|-   \frac{C_2}{f_c}\Bigr) + \frac{C_2^2}{f_c^2}\,\mathds{1}_{(0,+\infty)}\Bigl(|\calI_k|-  \frac{C_2}{f_c}\Bigr) \right)\\
            & \leq 1 - \frac{C_1 f_c^2}{2K+1} \sum_{k = -K}^{K} \min\bigl( \calI_k^2, C_2^2/f_c^2\bigr) \\
            & \leq 1 - C_1 f_c^2 \Delta_x^2,
        \end{align*}
        where the last inequality is precisely the stability condition \eqref{def:dynamicalStability}.

        Therefore we have that $\left\lVert P_{T^c} q \right\rVert_\infty \leq 1 - C_1 f_c^2 \Delta_x^2 \le 1 - C/\textnormal{SRF}_x^2$, for some absolute constant $C>0$. By inserting this bound into \eqref{eq:stabilityBound}, we conclude the proof.

    \end{proof}

\section{Numerical simulations}\label{sec:numerical}

\subsection{Methods}
Solving minimization problem \eqref{eq:lasso_v} in all its generality is not an easy task, since it is nonlinear and infinite dimensional. It is possible to use an analogue discrete problem, where the locations and velocities are fixed on a grid whose size determines the resolution we want to obtain. However, this methods becomes intractable for a fine resolution. 

In this work, we seek to validate our approach using a reconstruction method in the continuum. In~\cite{candes2014towards}, an algorithm with solutions in the continuum is presented in the one dimensional case, but we require a method for  higher dimensions. In~\cite{boyd2015alternating}, the authors develop an algorithm  to solve the following problem for any linear operator $\mathcal{F}$ from the space of positive punctual measures to $\mathbb{R}^n$:
\begin{equation}
    \label{eq:min-alt}
    \min_{\tilde{\mu}} \Vert \mathcal{F} \tilde{\mu} - Y \Vert_2^2 \quad \textrm{subject to } \Vert \tilde{\mu} \Vert_{TV} \leq M.
\end{equation}
Albeit the proposed algorithm is limited to positive weights, this is a realistic expectation in the case of many physical signals, for instance those generated by micro-bubbles in ultrafast ultrasound imaging. Clearly, a minimizer of \eqref{eq:min-alt} will be a minimizer of \eqref{eq:lasso_v} provided that the initial total variation is known.
\begin{lem}
    Assume that $\mu$ is the unique solution of \eqref{eq:lasso_v}. Then $\mu$ is the unique solution of \eqref{eq:min-alt} with $M = \lVert \mu \rVert_{TV}$.
\end{lem}
\begin{proof}
    Since \eqref{eq:lasso_v} admits a unique minimizer, every $\tilde{\mu} \neq \mu$ such that $\mathcal{F} \tilde{\mu} = Y$ verifies $\Vert {\tilde{\mu }} \Vert_{TV} > M$. Therefore, $\tilde{\mu}$ is the unique minimizer of \eqref{eq:min-alt}.
\end{proof}

The codes of the simulations of this paper are available at \url{https://github.com/panchoop/dynamic_spike_super_resolution}.
\subsection{The measurements}

We consider  the perfect low-pass filter described in \Cref{subsec:lowFreq} with $d=1$ as forward measurement operator, where the measured Fourier frequencies are $\{-f_c, \dots, f_c \}$ for some $f_c \in\N$ and the number of time samples are $2K+1$ with sampling rate $\tau$. The considered parameters for the simulations are:
\begin{equation}
    f_c = 20, \quad
    K = 2, \quad  
    \tau = 0.5.
    \label{eq:params}
\end{equation} 
Each simulation contains a random number of particles, between 4 and 10. Furthermore each particle is generated uniformly in $\Omega$ and has an associated random weight taken uniformly between $0.9$ and $1.1$. In total, we made $18,000$ simulations.

In order to validate our dynamic spike super-resolution approach, we compare  dynamical and static reconstructions and study the stability of our setting. Dynamical reconstruction refers to the recovery of positions and velocities in the space-velocity domain $\Omega \subset \R^2$, whereas by static reconstruction we mean  the recovery of the positions of the particles in $[0,1]$  at some fixed time.

Since we are simulating and reconstructing particles' locations and velocities in the continuum, there are natural associated numerical errors and we require a criterion to call a reconstruction either a success or a failure. Let $\Delta_x,\Delta_v, \Delta_w > 0$ be the maximal accepted errors in the reconstruction of position, velocity and weight, respectively. More precisely, if we consider a single particle $w_i \delta_{(x_i,v_i)}$ and its associated reconstruction $\tilde w_i \delta_{(\tilde x_i, \tilde v_i)}$, then we say that the particle was successfully reconstructed if 
\begin{equation*}
    \lvert x_i - \tilde x_i \rvert \leq \Delta_x, \quad \lvert v_i - \tilde v_i \rvert \leq \Delta_v, \quad \lvert w_i - \tilde w_i \rvert \leq \Delta _w.
\end{equation*}
We consider a configuration successfully reconstructed if, for each particle, there exists a unique successfully reconstructed particle and there are no additional reconstructed particles. These criteria are analogue for static reconstructions.

Notice that in the case of space and velocity, the error values scale with respect to the maximal imaging frequency ($f_c$ in space and $K f_c \tau$ in velocity), thus it is natural to consider the super-resolution factors
    \begin{equation*}
        \text{SRF}_x = \frac{1}{f_c} \frac{1}{\Delta_x} \qquad \text{and} \qquad \text{SRF}_v = \frac{1}{f_c K \tau} \frac{1}{\Delta_v},
    \end{equation*}
which will be set appropriately in each simulation.

Let us introduce the following measure of separation of a configuration of particles $T=\{(x_i,v_i)\}_i\subseteq \Omega$:
\begin{equation} \label{eq:dynsep}
    \Delta_{dyn}(T) = \max^3_{k \in\{-K,\dots,K\}} \min_{i\neq j} \lvert x_i - x_j + \tau k ( v_i - v_j) \rvert,
\end{equation}
where $\displaystyle \max^3$ is the third highest element of the set. The quantity $\Delta_{dyn}$ represents condition (1) of \Cref{thm:main}, for any subset of measurements $\calK \subseteq \{-K,\ldots,K\}$ with $\lvert \calK \rvert = 3$. We will evaluate the simulated reconstructions against this measure of separation, which will be scaled by $\frac{1}{f_c}$, as the theoretical allowed resolution depends on this value (see \Cref{rem:thm}a). Condition (2)  of \Cref{thm:main} on the absence of ghost particles could also be considered, but a formula for that purpose is extremely complicated and in view of \Cref{prop:ghost} we believe it is not necessary to include it for the following analysis.
\subsection{Results}
\subsubsection{Comparison of dynamical and static reconstructions}
\label{sec:noiseless}
For a given configuration  of particles, we consider three reconstruction procedures.
\begin{itemize}
    \item The dynamical reconstruction: we take the whole data and recover both positions and velocities of all particles.
    \item The static reconstruction: we perform static reconstruction of the positions at each time step independently, and we call it a success if the reconstruction is successful for at least one time step. 
    \item The static 3 reconstruction: we perform static reconstruction of the positions at each time step independently, and we call it a success if the reconstruction is successful for at least three time steps. The rationale of this case is that with three successful static reconstructions, it would be possible to  use some tracking technique to reconstruct the velocities afterwards. This case also encodes the theoretical limit given by \Cref{thm:main}. 
\end{itemize} 

\begin{figure}
    \centering
    \InputImage{0.8\textwidth}{0.45\textwidth}{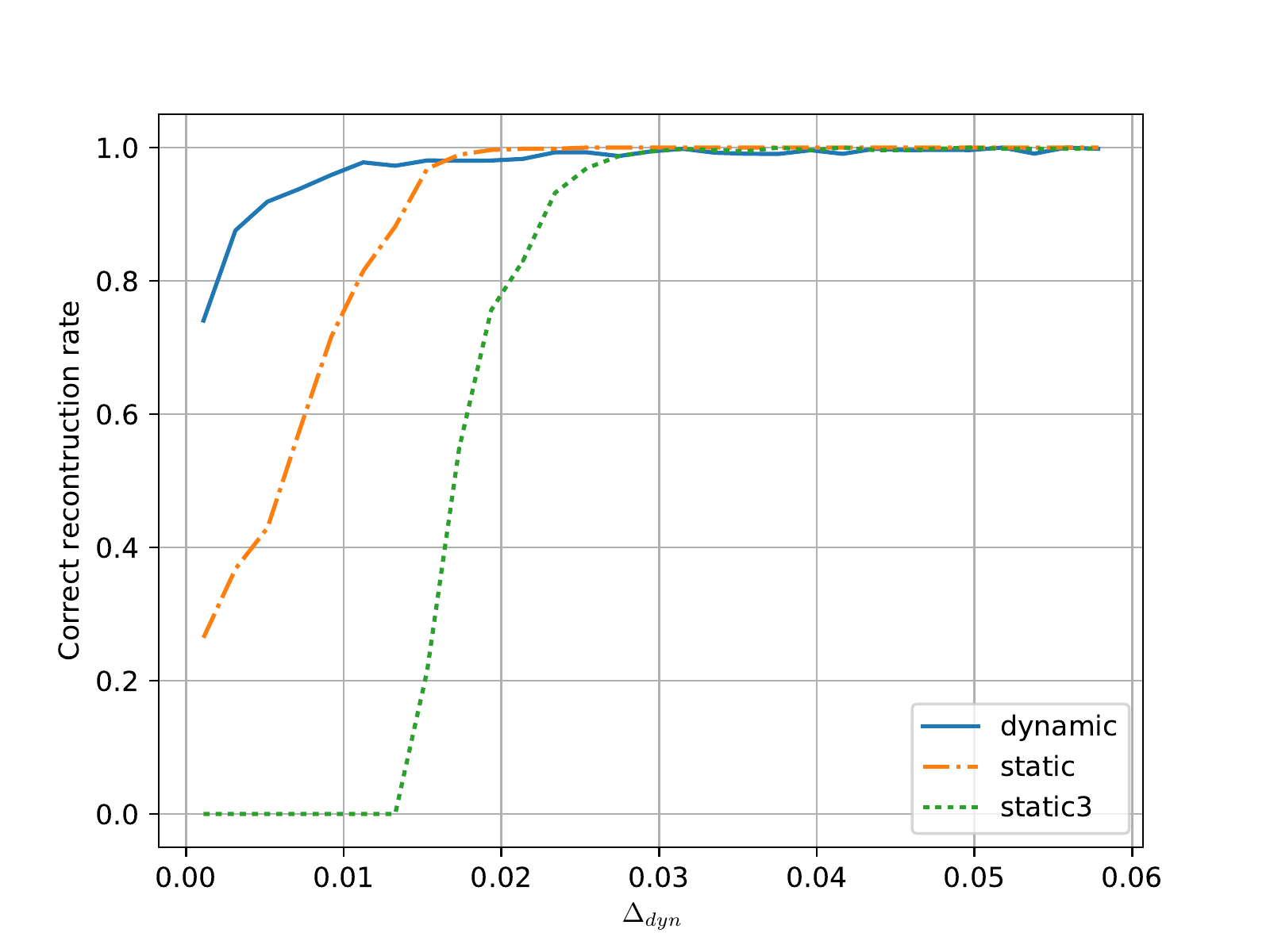}
    \caption{Successful reconstruction rates for the cases described in \Cref{sec:noiseless}, with $\text{SRF}_x=\text{SRF}_v = 1000$ and $\Delta_w = 0.01$. In the horizontal axis  we consider the measure of separation \eqref{eq:dynsep} scaled by $\nicefrac{1}{f_c}$.}
    \label{fig:noiseless}
\end{figure}

In \Cref{fig:noiseless} we present the rates of successful reconstructions for each of the three cases for randomly generated particles. We observe that the dynamical reconstruction has a much higher reconstruction rate than the static reconstructions for small values of $\Delta_{dyn}$, namely for configurations of particles that are never well-separated (this does not contradict \Cref{rem:thm}c, since stability deteriorates for close particles \cite{denoyelle2016support}, and so we face numerical errors even in the absence of noise). This gives a big advantage of dynamical  over static reconstructions, and shows that the assumptions of \Cref{thm:main} are in fact too strict: in practice, successful recovery happens much more often than the current theory predicts. For larger values of $\Delta_{dyn}$, the dynamical  reconstruction rate remains slightly below 1: this could be  explained by the presence of ghost particles, or by numerical issues of the minimization algorithm.

\subsubsection{Robustness to noise} \label{sec:noise}

We now study the stability of the dynamical reconstruction method and compare it to the stability for the static approaches. We consider a measurement noise model, given by a normally distributed noise scaled by a factor $\alpha \geq 0 $. More precisely, for a frequency $l$ and time sample $k$, our measurements are
\begin{equation*}
    \left( \mathcal{G}\omega \right)_{l,k} = \sum_{i=1}^N w_i e^{-i2\pi l (x_i + k \tau v_i)} + \alpha \left( \mathcal{N}_{l,k,1}  + i \mathcal{N}_{l,k,2} \right),
\end{equation*}
where $\mathcal{N}_{l,k,j}$ are independent standardized normal random variables. In  \Cref{fig:noise-dyn} we plot the reconstruction rate for different values of $\alpha$. To understand the noise level, we remind that $w_i \in (0.9, 1.1)$. The desired super-resolution factors and the weight threshold are larger in these experiments, since in the presence of noise we do not expect infinitely resolved reconstructions.

\begin{figure}
    \centering{}\captionsetup[subfigure]{width=150pt}
    \subfloat[Successful dynamical reconstruction rates under different intensities of noise.]{\label{fig:noise-dyn}
        \InputImage{0.48\textwidth}{0.4\textwidth}{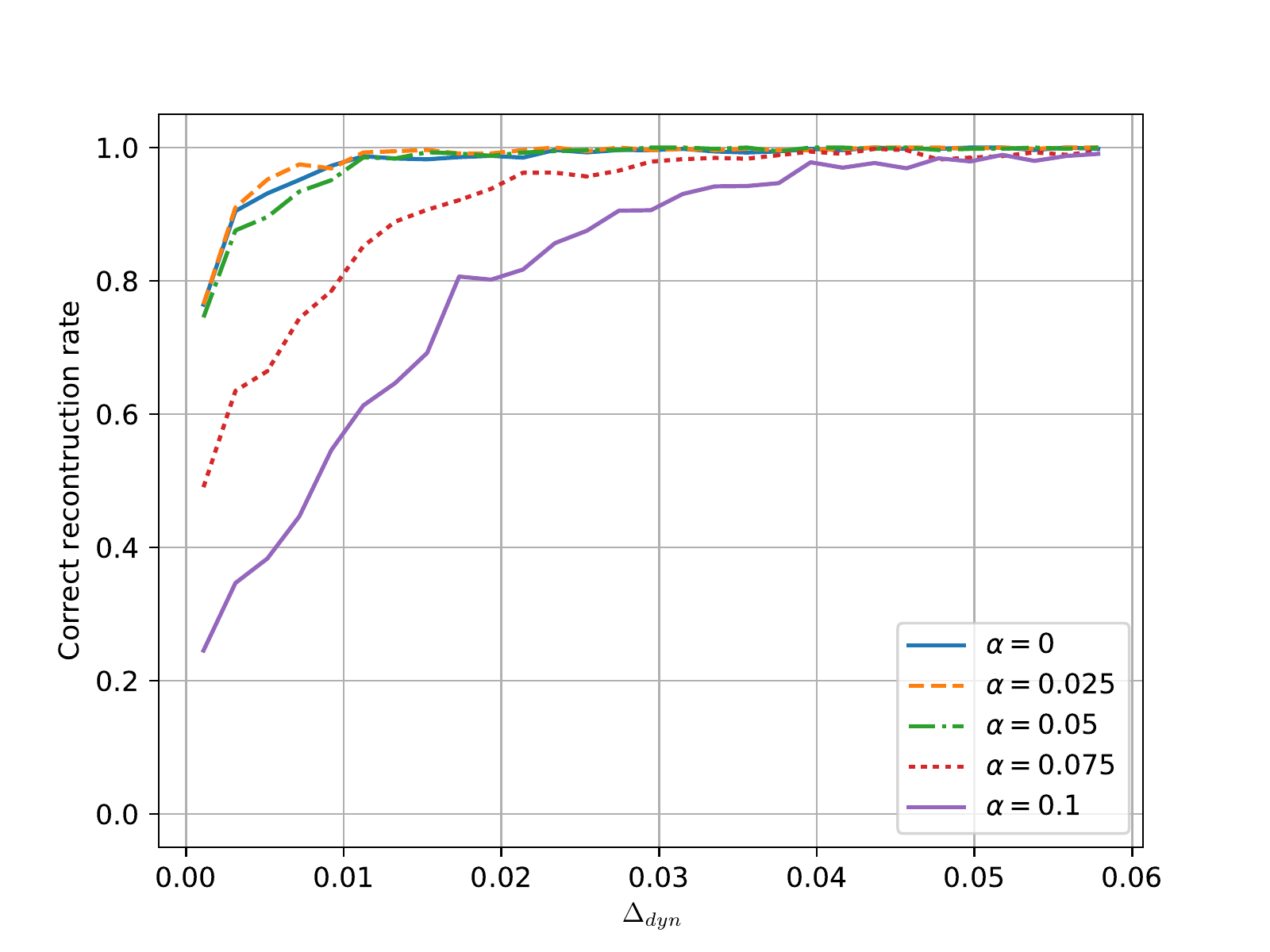}
    }
    \hfill{}
    \subfloat[Successful reconstruction rates for the cases described in \Cref{sec:noiseless}, with a fixed noise level $\alpha = 0.075$.]{\label{fig:noise-comp}
        \InputImage{0.48\textwidth}{0.4\textwidth}{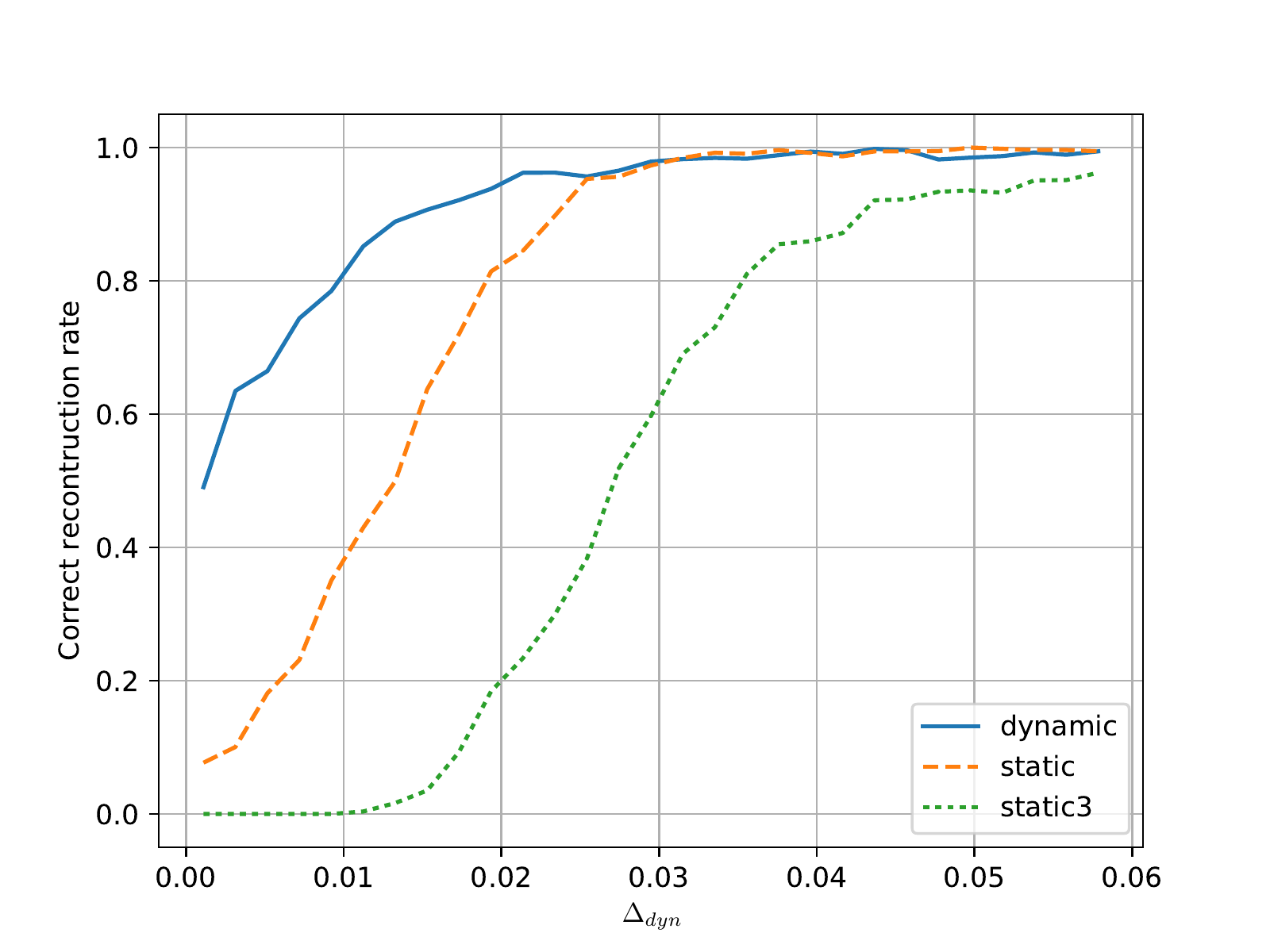}
    }
    \caption{Reconstruction rates for the measurement noise model described in \Cref{sec:noise}, with $\text{SRF}_x = \text{SRF}_v = 40$ and $\Delta_w = 0.05$.
    }
\end{figure}

We observe that the dynamical reconstruction method is stable to measurement noise, and this stability is similar to the one of the static reconstruction, as we can see in \Cref{fig:noise-comp}.

\subsubsection{Robustness to curvature of trajectory}\label{sub:curvature}

We study how the reconstruction algorithm fares when instead of a constant velocity, we consider a curved trajectory for the imaged particles. For this purpose we consider the following dynamics for a particle $\delta_{x_i,v_i}$:
\begin{equation*}
    x(t) = x_i + v_i t + \frac{a}{2} t^2 = x_i + v_i t (1 _+ \frac{a}{2 v_i}t).
\end{equation*}
We consider $\beta = \frac{a}{2 v_i}\tau K$ as measure of curvature.
In \Cref{fig:curvatureEX} we present one example of a considered curvature for a trajectory, and in  \Cref{fig:curvatureDyn} we present the dynamical reconstruction rates for different values of $\beta$.

\begin{figure}
    \centering{}\captionsetup[subfigure]{width=150pt}
    \subfloat[Comparison of the trajectories of a particle with speed $0.25$ and  curvature $\beta = 0$ and $\beta = 0.03$.]{\label{fig:curvatureEX}
        \InputImage{0.48\textwidth}{0.4\textwidth}{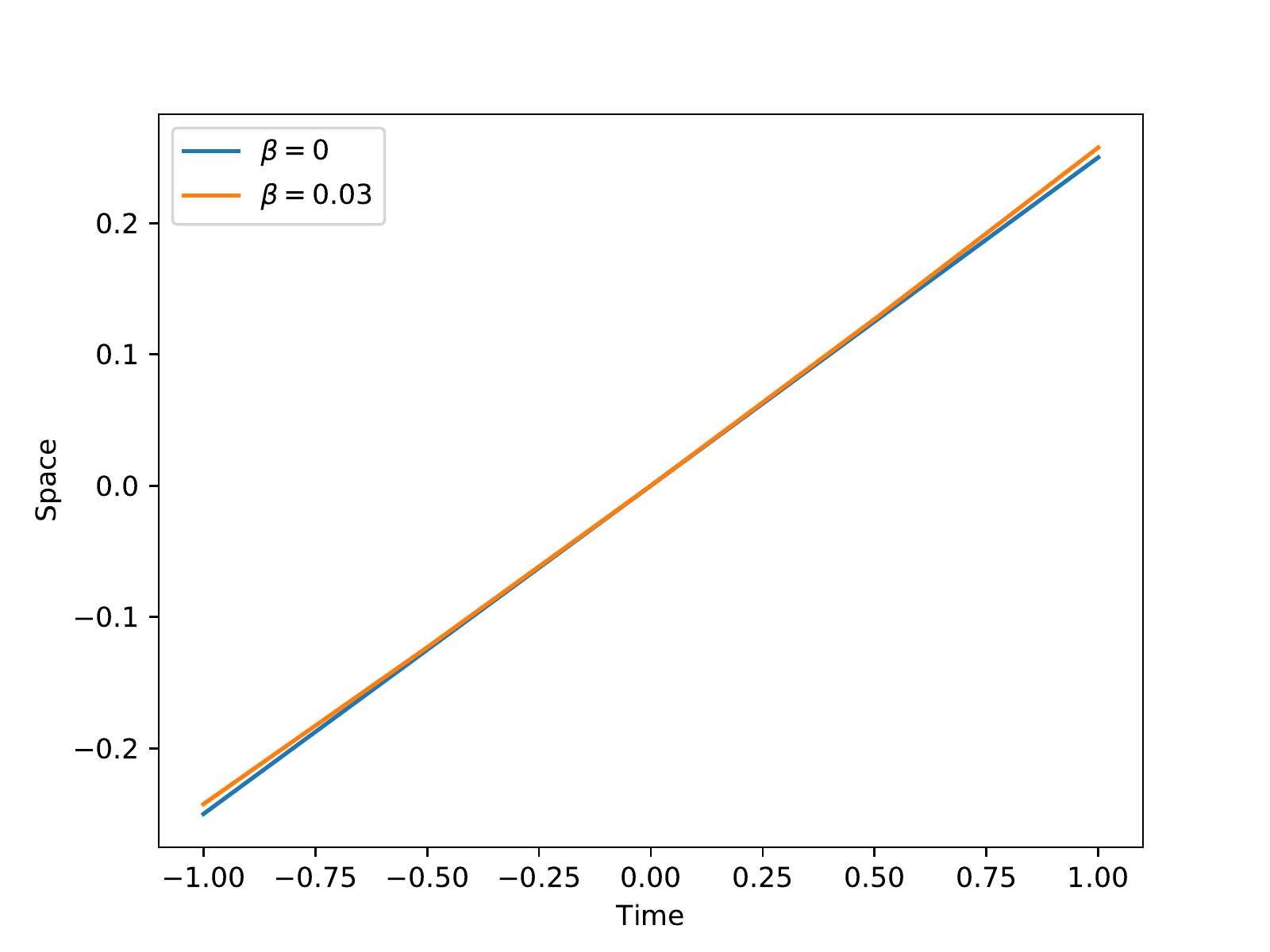}
    }
    \hfill{}
    \subfloat[Successful dynamical reconstruction rates for different levels of curvature.]{\label{fig:curvatureDyn}
        \InputImage{0.48\textwidth}{0.4\textwidth}{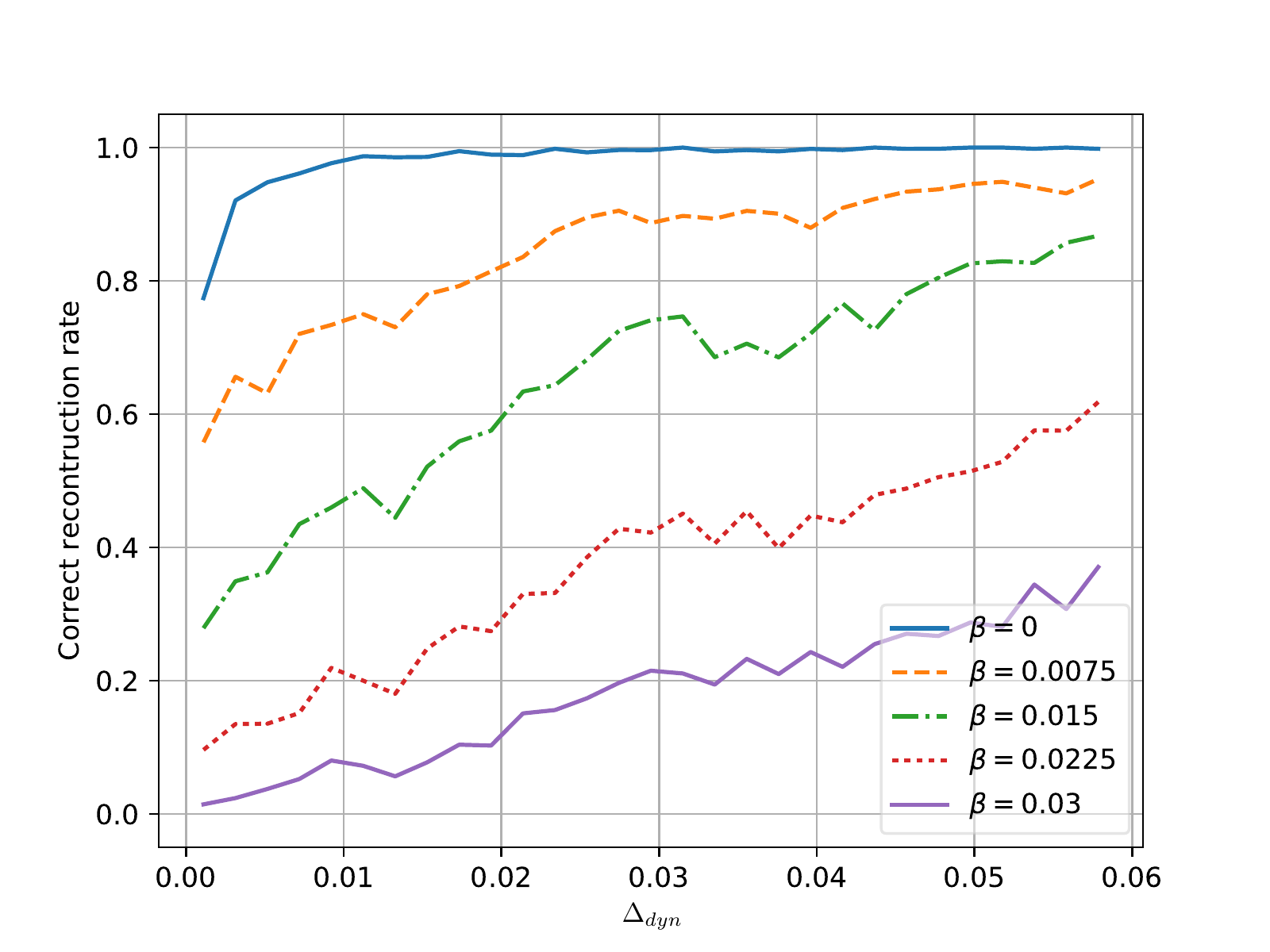}
    }
    \caption{Reconstruction rates with the curvature model described in \Cref{sub:curvature}, with $\text{SRF}_x=\text{SRF}_v=1$ and $\Delta_w = 0.2$.
    }
\end{figure}

For the presented examples, we relaxed the super resolution factors to $1$, since the algorithm is reconstructing some position and velocity, but reasonably it is not located exactly in the target $(x_i,v_i)$; similarly, we relaxed the weight threshold condition. Nonetheless, we notice that the stability of the method with respect to changes in the  curvature of the trajectories is quite poor (even though probably sufficient in a realistic context, see $\S$\ref{sub:num}). This may be explained by a not optimal choice of the parameters of the minimization algorithm (which should take into account that the forward model is not exact) and by the difficulties for a linear model to capture nonlinear movements. In any case, higher order models are expected to solve this issue and represent an interesting direction for future research.

\section{Applications to ultrafast ultrasonography}\label{sec:ultrafast}
In this section, we describe a protocol to apply our method to the problem of super-resoluted imaging of blood vessels arising from ultrafast ultrasonography, as mentioned in the Introduction. The setting is the following: we have a sequence of images of a medium containing blood vessels, in which point reflectors are randomly activated during a small time interval. These reflectors are moving inside the blood vessels and their velocity is approximately the same as that of blood in the blood vessels.

We assume that we can filter out clutter signal coming from other sources than these reflectors. The recorded images are then convolutions of these point sources by the point spread function  (PSF). The PSF of 
 ultrafast ultrasound imaging was derived in \cite{alberti2017mathematical}, but filtering out clutter signal changes the shape of the PSF, which is here approximated by a Gaussian function. However, the analysis of \cite{alberti2017mathematical} allows for a precise derivation, which we leave for future investigation.

\subsection{Fully automated imaging protocol}

We must  make sure that the reconstruction algorithm works in this sequence, and that reflectors do not appear or disappear   in a chosen sequence. Given all these remarks, we propose the following imaging procedure, which fully automatically produces a super-resoluted image of blood vessels with velocities. It is composed of two steps:

\begin{figure}
    \centering
    \InputImage{0.95\textwidth}{0.4\textwidth}{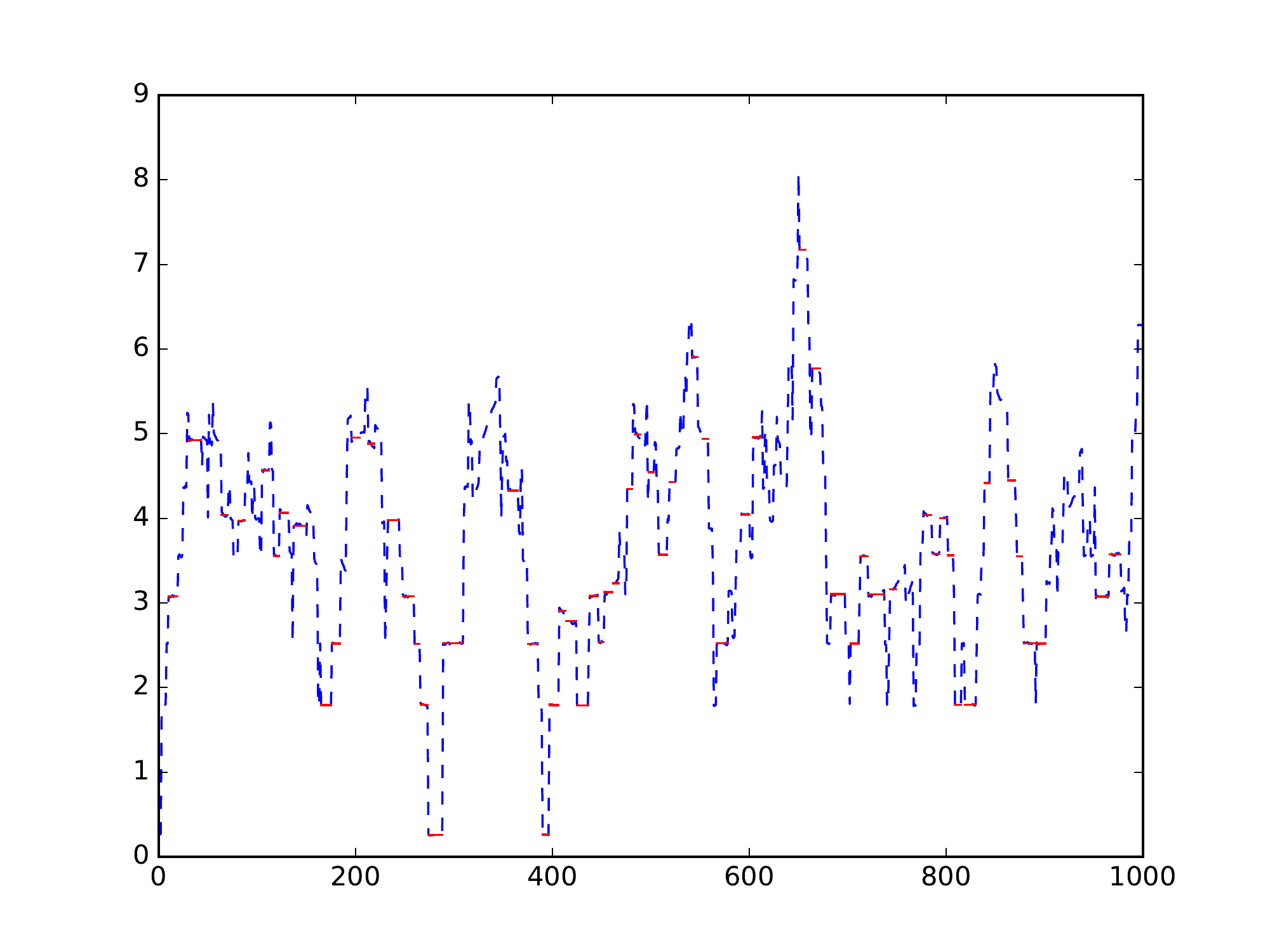}
    \caption{Graph of the $\ell^2$ norm of the simulated measurements in each frame. \label{fig:norm}}
\end{figure}

\paragraph{Choosing reconstruction intervals.} In order to apply our dynamical reconstruction algorithm to this problem, we have to select consecutive frames during which the particles do not appear or vanish from one frame to another. One way to ensure this condition is to select time intervals during which the $\ell^2$ norm of the observations is constant. Particles appearing or vanishing make the $\ell^2$ norm jump, whereas close particles will make the $\ell^2$ norm vary slowly. \Cref{fig:norm} illustrates the variation of $\ell^2$ norm as a function of frame number (dashed line), and the selected intervals (red solid line), for the case presented in the numerical experiments below.

\paragraph{Reconstructing position and velocities.} We then propose to reconstruct positions and velocities using the algorithm presented in the previous sections, using the PSF of ultrafast ultrasound, in each of the intervals chosen in the previous step. By aggregating all positions obtained using this algorithm, we obtain a super-resoluted image of the blood vessels.

\subsection{Numerical experiments}
\label{sub:num}

\begin{figure}
	        \centering
	        \subfloat[Diagram describing the considered vessels with their respective blood flows.]{\includegraphics[width=0.39\textwidth]{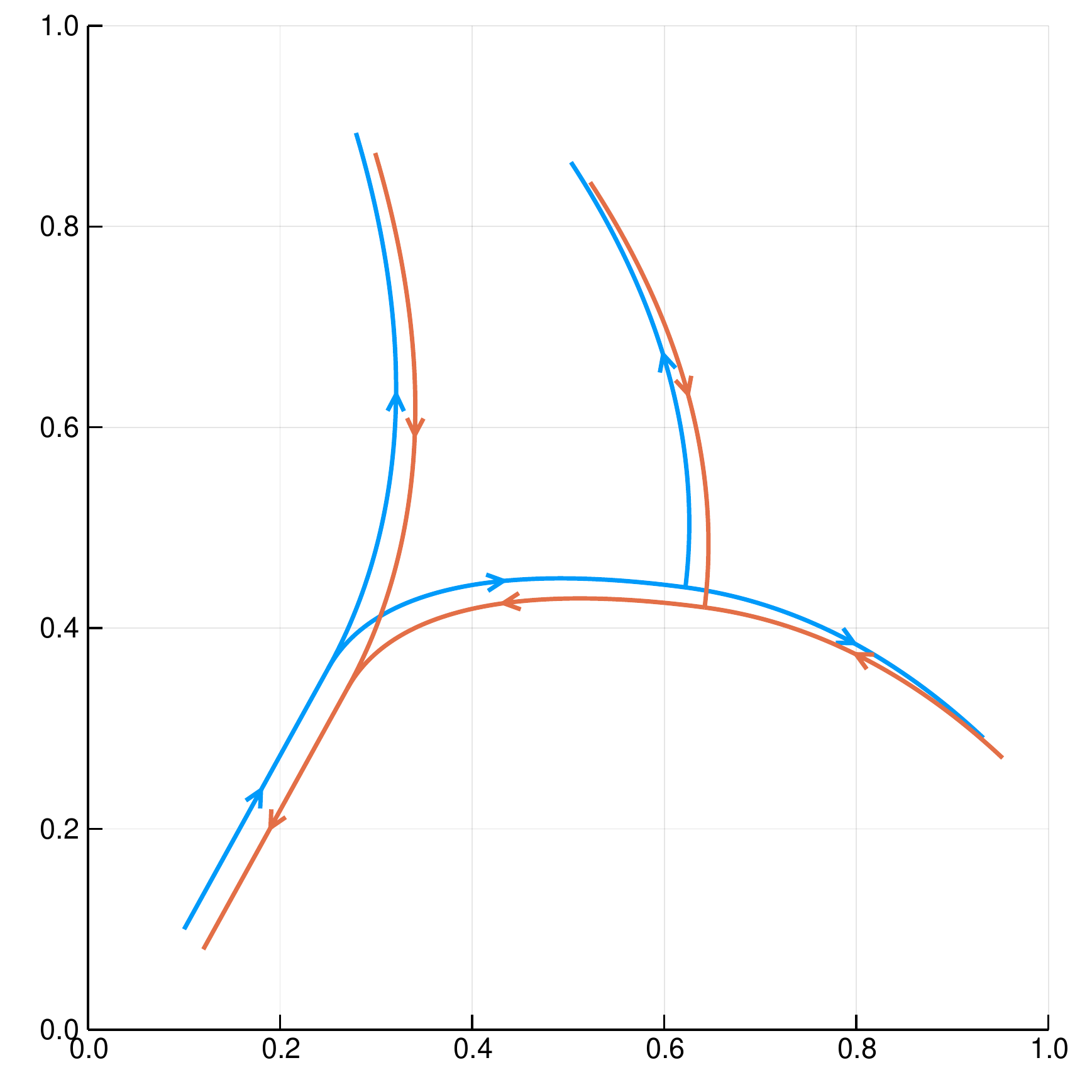}\label{fig:vessel_scheme}}
	        \hfill
	        \subfloat[B-mode image of the vessels, taken over 2 seconds of measurements.]{\InputImage{0.45\textwidth}{0.45\textwidth}{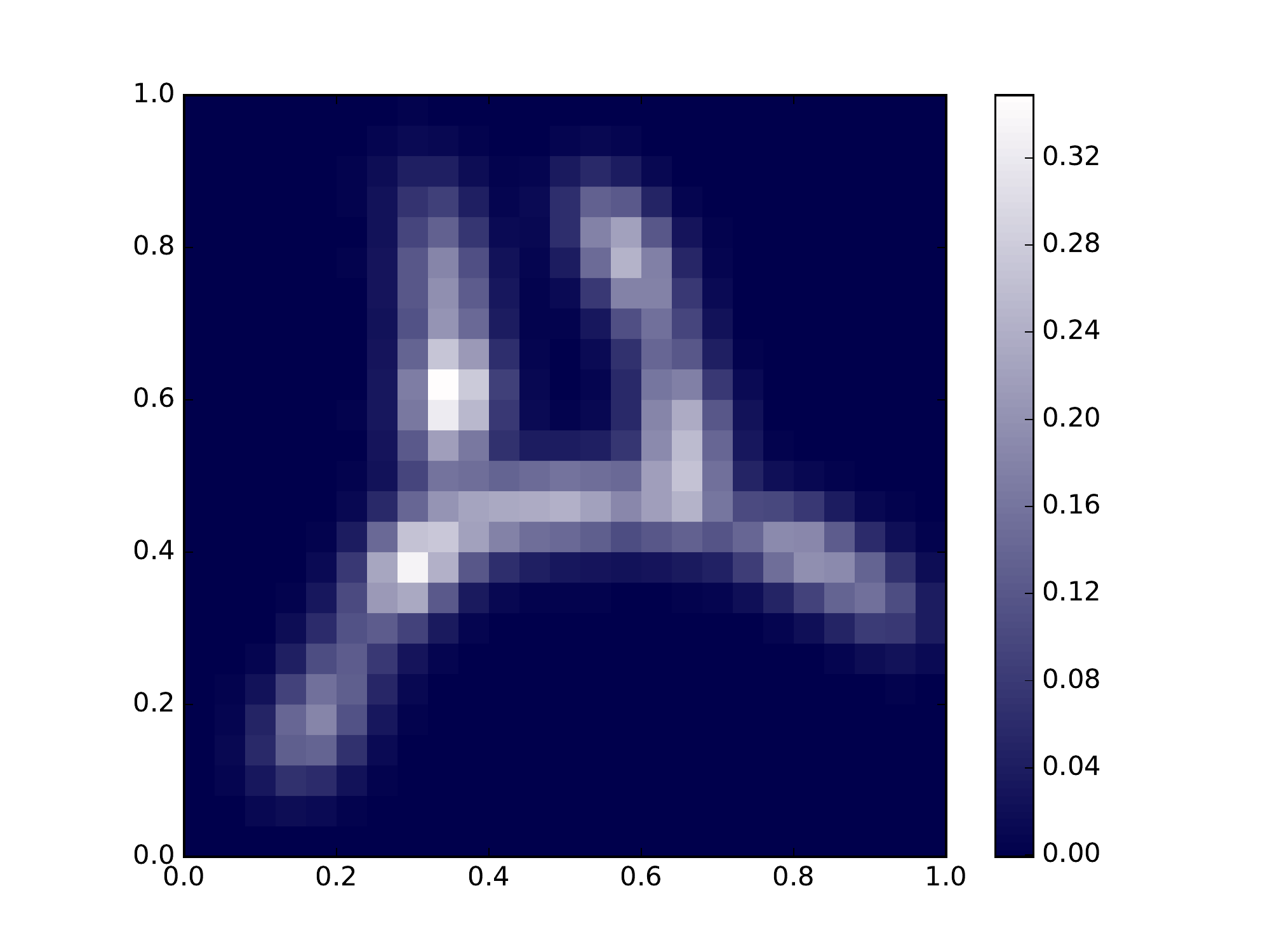}\label{fig:bmode}}
	        \caption{Simulated blood vessels}
	    \end{figure}

\subsubsection{Setting}
In order to test this procedure, we generate images using a toy example  in which we have two slightly separated curved blood vessels with two branches and opposite blood flows. In Figure \ref{fig:vessel_scheme} there is an illustration of the considered geometry. The velocities of the particles have constant modulus and the particles' weights are set to 1. The considered parameters are such that they are similar to the experimental ones:
\begin{itemize}[itemsep=1pt]
    \item Domain size: $\unit[1]{mm} \times \unit[1]{mm}$.
    \item Pixel size: $\unit[0.04]{mm} \times \unit[0.04]{mm}$.
    \item Maximal vessel separation: $ \sim \unit[0.03]{mm}$.
    \item Point spread function:$\ (x,y) \mapsto e^{-(x^2 + y^2)/2\sigma^2}, \ \ \sigma = 0.04.$
    \item Particles speed: $\unit[2]{mm/s}$.
    \item Sampling rate: $\tau = \unit[0.002]{s^{-1}}$
    \item Total acquisition time: $\unit[2]{s}$.
\end{itemize}

Since in ultrafast ultrasound microbubble imaging the particles are activated and deactivated randomly, we simulate this behavior in the following fashion. The activation  of a single particle is modeled as a Bernoulli random variable on each time sample, whereas the deactivation time is modeled as a Poisson random variable. Further, we include measurement noise as in \Cref{sec:noise}, with $\alpha = 0.01$. 

\subsubsection{Results}

\begin{figure}\centering
            \subfloat{\adjustbox{trim=0.3cm}{
                \InputImage{0.4\textwidth}{0.4\textwidth}{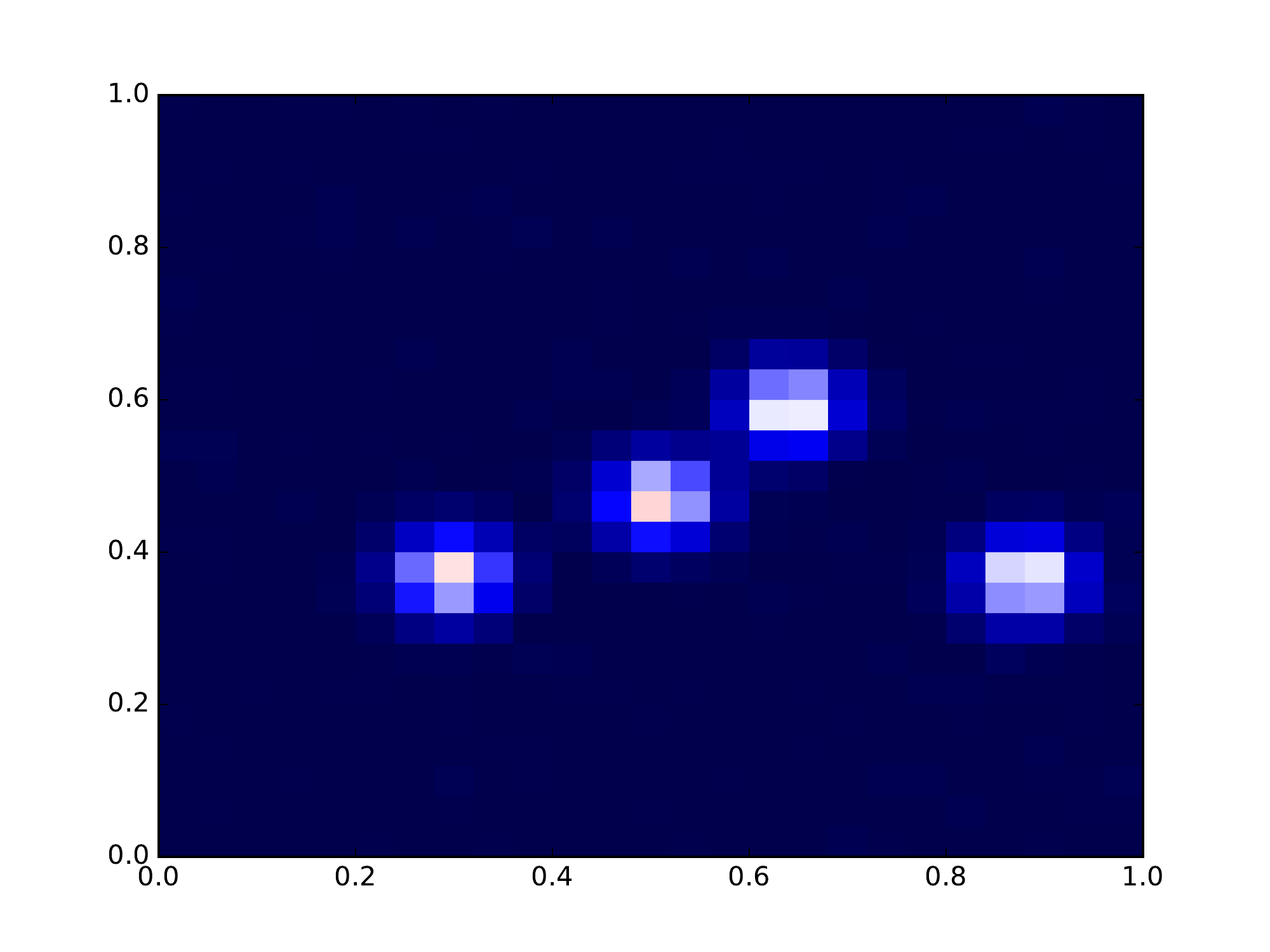}
        }} \hfill
        \subfloat{\adjustbox{trim=0.3cm}{
            \InputImage{0.4\textwidth}{0.4\textwidth}{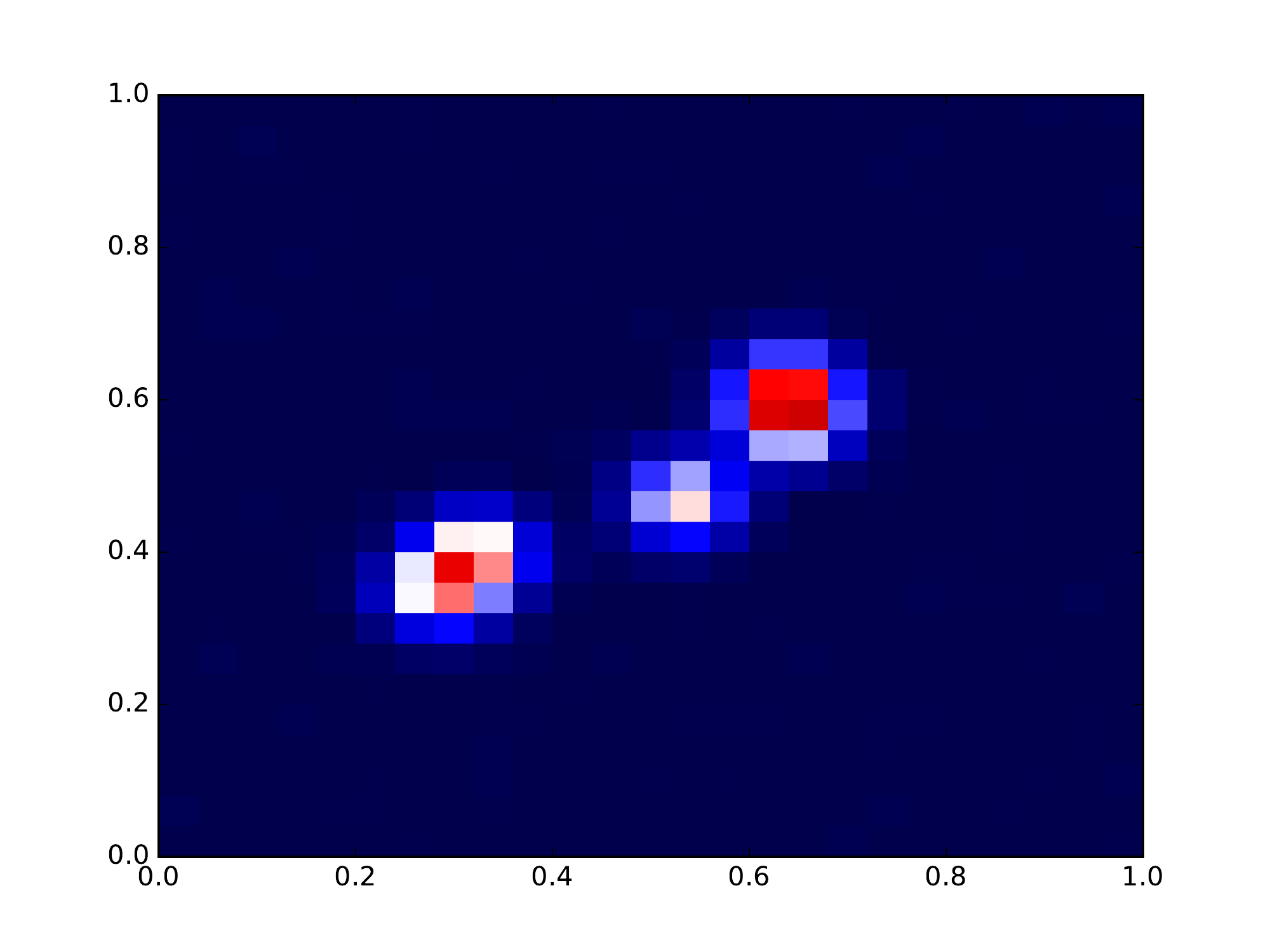}
        }} \hfill
        \subfloat{\adjustbox{trim=0.3cm}{
            \InputImage{0.4\textwidth}{0.4\textwidth}{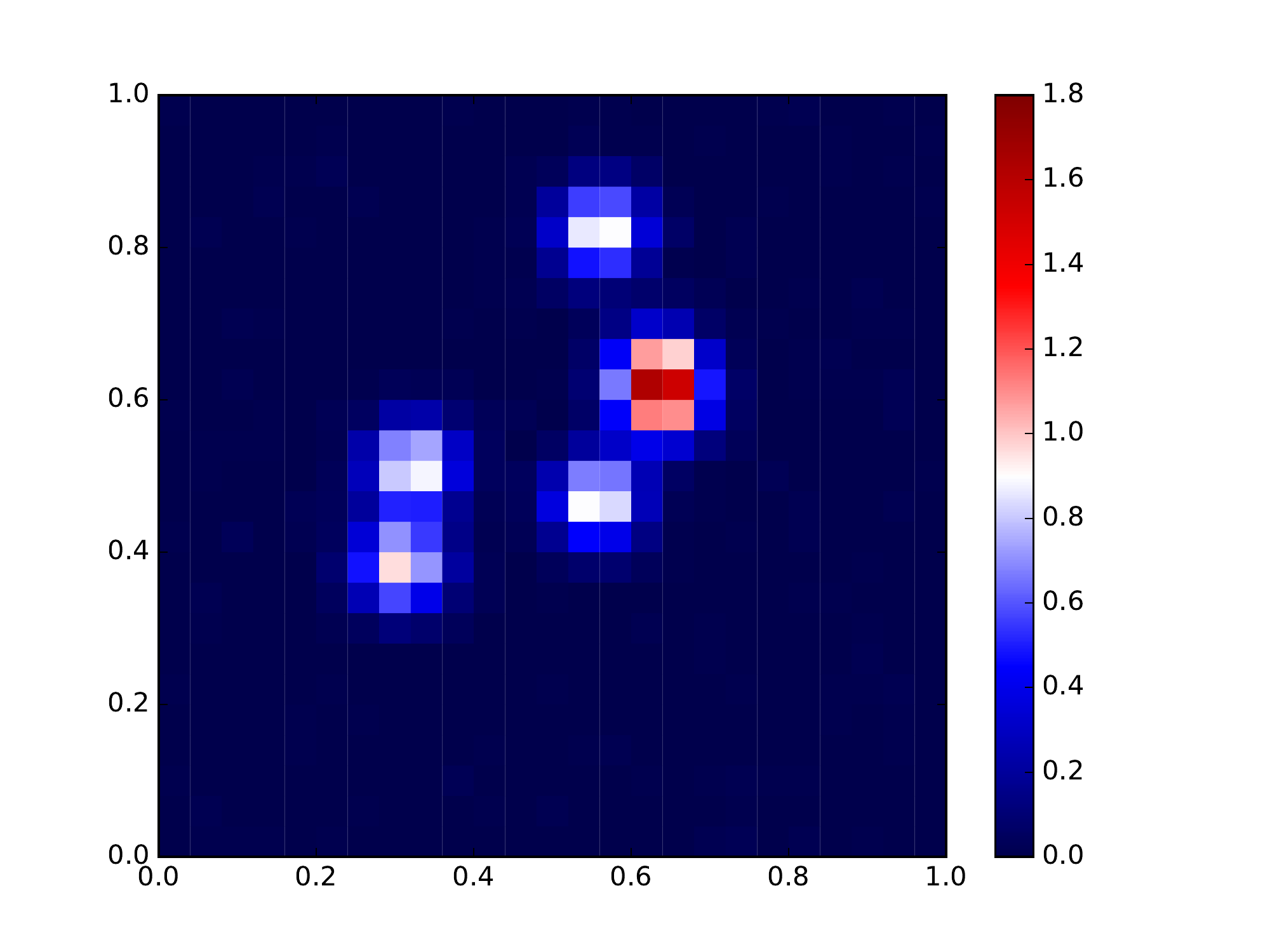}
        }}  \\
        \caption{Three different measurements with ultrafast ultrasound of the simulated particles.}
        \label{fig:threeframes}
    \end{figure}
    
    \begin{figure} 
	    	\centering
		    \includegraphics[width=0.7\textwidth]{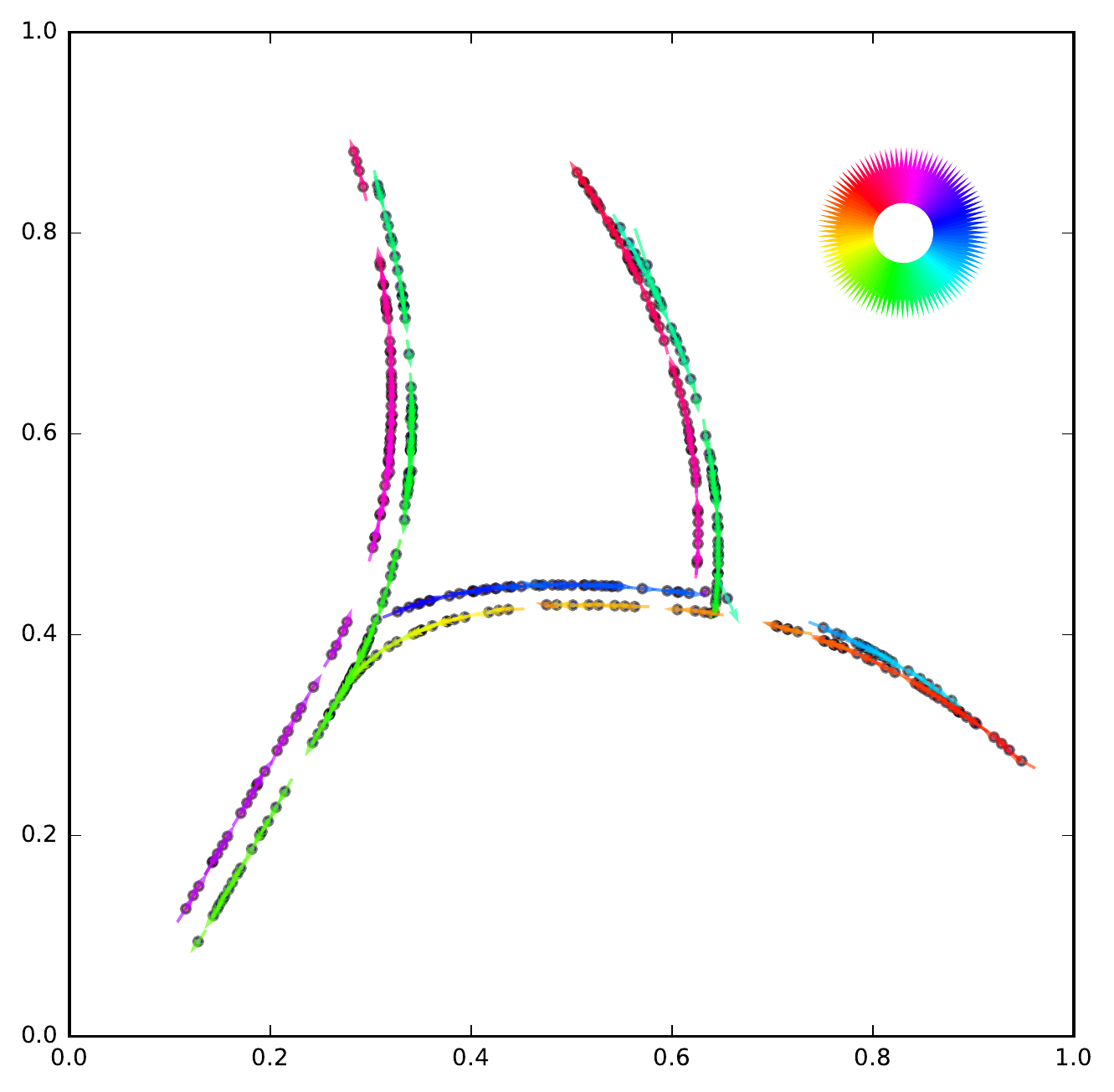}
	    	\caption{Super resolved reconstruction of the simulated particles,  with colored arrows representing the reconstructed velocities.}
	    	\label{fig:superres}
	    \end{figure}

In \Cref{fig:norm} we present the $\ell^2$ norms of the measurements at each time step. As an illustration, in \Cref{fig:threeframes} we show three frames of the simulated measurements. In \Cref{fig:bmode} we present the B-mode image, i.e.\ the average signal intensity over all the frames, which shows that the resolution does not allow for separation of the  vessels. \Cref{fig:superres} presents the reconstructed position and velocities  using the algorithm described in this work.  
Albeit a few errors are present,   positions and  velocities are reconstructed very accurately, even with particles that displace with non constant velocity.

    \section{Conclusion}\label{sec:conclusion}
    In this paper, we have introduced and studied a new framework for dynamical super-resolution imaging, which allows for super-resolved recovery of positions and velocities of particles from low-frequency measurements. The presented theoretical results  are validated by extensive simulations, related to low-frequency one-dimensional Fourier measurements and to two-dimensional ultrafast ultrasound localization microscopy. 
    
    In fact, the numerical experiments show that this approach works much better than what the current theory predicts, and so there is a need of further theoretical investigation. Indeed, the arguments presented in this paper are still based on the static reconstruction, while ideally one should consider the dynamical problem directly. Further, it would be nice to generalize this method to higher order models, in order to relax the assumption of linear trajectories.

    \bibliographystyle{plain}
    \bibliography{superres}

\begin{thebibliography}{10}

\bibitem{nobel}
{The Nobel Prize in Chemistry 2014}.
\newblock
  \url{http://www.nobelprize.org/nobel_prizes/chemistry/laureates/2014/}.
\newblock Nobel Media AB 2014. Web. 14 Feb 2018.

\bibitem{alberti2017mathematical}
Giovanni~S. Alberti, Habib Ammari, Francisco Romero, and Timoth\'ee Wintz.
\newblock Mathematical analysis of ultrafast ultrasound imaging.
\newblock {\em SIAM J. Appl. Math.}, 77(1):1--25, 2017.

\bibitem{ammari2008introduction}
Habib Ammari.
\newblock {\em An introduction to mathematics of emerging biomedical imaging},
  volume~62.
\newblock Springer, 2008.

\bibitem{ammari2017sub}
Habib Ammari, Brian Fitzpatrick, David Gontier, Hyundae Lee, and Hai Zhang.
\newblock Sub-wavelength focusing of acoustic waves in bubbly media.
\newblock {\em Proc. R. Soc. A}, 473(2208):20170469, 2017.

\bibitem{ammari2015mathematical}
Habib Ammari and Hai Zhang.
\newblock A mathematical theory of super-resolution by using a system of
  sub-wavelength helmholtz resonators.
\newblock {\em Communications in Mathematical Physics}, 337(1):379--428, 2015.

\bibitem{ammari2015super}
Habib Ammari and Hai Zhang.
\newblock Super-resolution in high-contrast media.
\newblock {\em Proc. R. Soc. A}, 471(2178):20140946, 2015.

\bibitem{spike-acha-2015}
Jean-Marc Aza\"\i~s, Yohann de~Castro, and Fabrice Gamboa.
\newblock Spike detection from inaccurate samplings.
\newblock {\em Appl. Comput. Harmon. Anal.}, 38(2):177--195, 2015.

\bibitem{bendory2015robust}
T.~Bendory.
\newblock Robust recovery of positive stream of pulses.
\newblock {\em IEEE Transactions on Signal Processing}, 65(8):2114--2122, April
  2017.

\bibitem{Betzig2006}
Eric Betzig, George~H. Patterson, Rachid Sougrat, O.~Wolf Lindwasser, Scott
  Olenych, Juan~S. Bonifacino, Michael~W. Davidson, Jennifer
  Lippincott-Schwartz, and Harald~F. Hess.
\newblock Imaging intracellular fluorescent proteins at nanometer resolution.
\newblock {\em Science}, 2006.

\bibitem{boyd2015alternating}
Nicholas Boyd, Geoffrey Schiebinger, and Benjamin Recht.
\newblock The alternating descent conditional gradient method for sparse
  inverse problems.
\newblock In {\em Computational Advances in Multi-Sensor Adaptive Processing
  (CAMSAP), 2015 IEEE 6th International Workshop on}, pages 57--60. IEEE, 2015.

\bibitem{candes2013}
Emmanuel~J. Cand\`es and Carlos Fernandez-Granda.
\newblock Super-resolution from noisy data.
\newblock {\em J. Fourier Anal. Appl.}, 19(6):1229--1254, 2013.

\bibitem{candes2014towards}
Emmanuel~J Cand{\`e}s and Carlos Fernandez-Granda.
\newblock Towards a mathematical theory of super-resolution.
\newblock {\em Communications on Pure and Applied Mathematics}, 67(6):906--956,
  2014.

\bibitem{de2012exact}
Yohann De~Castro and Fabrice Gamboa.
\newblock Exact reconstruction using beurling minimal extrapolation.
\newblock {\em Journal of Mathematical Analysis and applications},
  395(1):336--354, 2012.

\bibitem{demene2015spatiotemporal}
Charlie Demene, Thomas Deffieux, Mathieu Pernot, Bruno-Felix Osmanski, Valerie
  Biran, Jean-Luc Gennisson, Lim-Anna Sieu, Antoine Bergel, Stephanie Franqui,
  Jean-Michel Correas, et~al.
\newblock Spatiotemporal clutter filtering of ultrafast ultrasound data highly
  increases doppler and fultrasound sensitivity.
\newblock {\em Medical Imaging, IEEE Transactions on}, 34(11):2271--2285, 2015.

\bibitem{denoyelle2016support}
Quentin Denoyelle, Vincent Duval, and Gabriel Peyr\'e.
\newblock Support recovery for sparse super-resolution of positive measures.
\newblock {\em J. Fourier Anal. Appl.}, 23(5):1153--1194, 2017.

\bibitem{desailly2013sono}
Yann Desailly, Olivier Couture, Mathias Fink, and Mickael Tanter.
\newblock Sono-activated ultrasound localization microscopy.
\newblock {\em Applied Physics Letters}, 103(17):174107, 2013.

\bibitem{moerner1997}
Robert~M. Dickson, Andrew~B. Cubitt, Roger~Y. Tsien, and W.~E. Moerner.
\newblock On/off blinking and switching behaviour of single molecules of green
  fluorescent protein.
\newblock {\em Nature}, 388:355--358, 1997.

\bibitem{2017radial}
Charles Dossal, Vincent Duval, and Clarice Poon.
\newblock Sampling the {F}ourier transform along radial lines.
\newblock {\em SIAM J. Numer. Anal.}, 55(6):2540--2564, 2017.

\bibitem{errico2015ultrafast}
Claudia Errico, Juliette Pierre, Sophie Pezet, Yann Desailly, Zsolt Lenkei,
  Olivier Couture, and Mickael Tanter.
\newblock Ultrafast ultrasound localization microscopy for deep
  super-resolution vascular imaging.
\newblock {\em Nature}, 527(7579):499--502, 2015.

\bibitem{fernandez2013support}
Carlos Fernandez-Granda.
\newblock Support detection in super-resolution.
\newblock {\em arXiv preprint arXiv:1302.3921}, 2013.

\bibitem{Hell94}
Stefan~W. Hell and Jan Wichmann.
\newblock Breaking the diffraction resolution limit by stimulated emission:
  stimulated-emission-depletion fluorescence microscopy.
\newblock {\em Opt. Lett.}, 19(11):780--782, Jun 1994.

\bibitem{hess2006ultra}
Samuel~T Hess, Thanu~PK Girirajan, and Michael~D Mason.
\newblock Ultra-high resolution imaging by fluorescence photoactivation
  localization microscopy.
\newblock {\em Biophysical journal}, 91(11):4258--4272, 2006.

\bibitem{lemoult2013wave}
Fabrice Lemoult, Nad{\`e}ge Kaina, Mathias Fink, and Geoffroy Lerosey.
\newblock Wave propagation control at the deep subwavelength scale in
  metamaterials.
\newblock {\em Nature Physics}, 9(1):55, 2013.

\bibitem{lerosey2007focusing}
Geoffroy Lerosey, Julien De~Rosny, Arnaud Tourin, and Mathias Fink.
\newblock Focusing beyond the diffraction limit with far-field time reversal.
\newblock {\em Science}, 315(5815):1120--1122, 2007.

\bibitem{montaldo-tanter-bercoff-benech-finck-2009}
G.~Montaldo, M.~Tanter, J.~Bercoff, N.~Benech, and M.~Fink.
\newblock Coherent plane-wave compounding for very high frame rate
  ultrasonography and transient elastography.
\newblock {\em Ultrasonics, Ferroelectrics, and Frequency Control, IEEE
  Transactions on}, 56(3):489--506, March 2009.

\bibitem{2016-Morgenshtern-candes}
Veniamin~I. Morgenshtern and Emmanuel~J. Cand\`es.
\newblock Super-resolution of positive sources: the discrete setup.
\newblock {\em SIAM J. Imaging Sci.}, 9(1):412--444, 2016.

\bibitem{poon2017multi}
Clarice Poon and Gabriel Peyr{\'e}.
\newblock Multi-dimensional sparse super-resolution.
\newblock {\em arXiv preprint arXiv:1709.03157}, 2017.

\bibitem{62014-tanter-fink}
M.~Tanter and M.~Fink.
\newblock Ultrafast imaging in biomedical ultrasound.
\newblock {\em IEEE Transactions on Ultrasonics, Ferroelectrics, and Frequency
  Control}, 61(1):102--119, January 2014.

\bibitem{thompson2012extending}
Michael~A Thompson, Matthew~D Lew, and WE~Moerner.
\newblock Extending microscopic resolution with single-molecule imaging and
  active control.
\newblock {\em Annual review of biophysics}, 41:321--342, 2012.

\end{thebibliography}
\end{document}